\documentclass[preprint,12pt]{elsarticle}

\usepackage{algorithmic}
\usepackage{algorithm}
\usepackage{graphicx}
\usepackage{amsmath,amsfonts,amssymb}
\usepackage{epsfig,makecell,float}
\usepackage{setspace,mathrsfs}
\usepackage{tocloft}
\usepackage{textcomp}
\usepackage{multirow,indentfirst,times,color}
\usepackage[caption=false,font=footnotesize]{subfig}
\usepackage{booktabs}
\usepackage{threeparttable}
\usepackage{amsthm}
\usepackage{extarrows}
\usepackage{bm}

\usepackage[titletoc]{appendix}
\usepackage{epstopdf}
\usepackage{natbib}

\usepackage{todonotes}
\usepackage{hyperref}

\DeclareMathOperator*{\argmin}{arg\,min}
\DeclareMathOperator*{\argmax}{arg\,max}

\biboptions{numbers,sort&compress}
\newtheorem{thm}{Theorem}
\newtheorem{defn}{Definition}
\newtheorem{remark}{Remark}

\begin{document}

\begin{frontmatter}



\title{Sampling of Graph Signals Based on Joint Time-Vertex Fractional Fourier Transform} 

\author[a,b]{Yu Zhang} 
\author[a,b]{Bing-Zhao Li\corref{mycorrespondingauthor}}
\cortext[mycorrespondingauthor]{Corresponding author}\ead{li\_bingzhao@bit.edu.cn}

\affiliation[a]{organization={School of Mathematics and Statistics, Beijing Institute of Technology},
	city={Beijing},
	postcode={100081}, 
	country={China}}

\affiliation[b]{organization={Beijing Key Laboratory on MCAACI, Beijing Institute of Technology},
	city={Beijing},
	postcode={100081}, 
	country={China}}

\begin{abstract}
With the growing demand for non-Euclidean data analysis, graph signal processing (GSP) has gained significant attention for its capability to handle complex time-varying data. This paper introduces a novel sampling method based on the joint time-vertex fractional Fourier transform (JFRFT), enhancing signal representation in time-frequency analysis and GSP. The JFRFT sampling theory is established by deriving conditions for the perfect recovery of jointly bandlimited signals, along with an optimal sampling set selection strategy. To further enhance the efficiency of large-scale time-vertex signal processing, the design of localized sampling operators is investigated. Numerical simulations and real data experiments validate the superior performance of the proposed methods in terms of recovery accuracy and computational efficiency, offering new insights into efficient time-varying signal processing.
\end{abstract}




\end{frontmatter}

\section{Introduction}
\label{Intro}
With the increasing volume of complex data emerging from non-Euclidean spaces and irregular domains such as social networks, transportation systems, sensor networks, and biological connectomes, the field of graph signal processing (GSP) \cite{GFTlaplace,GFTadjacency1,GFTadjacency2,Gvertex,Goverview,Ghistory} has witnessed significant advancements in both theory and applications in recent years. GSP defines network elements as graph vertices and their interconnections as graph edges. Signals associated with these vertices are termed graph signals. To analyze and process these graph signals, numerous researchers have extended classical signal processing tools to the GSP domain, including the graph Fourier transform (GFT) \cite{GFTlaplace,GFTadjacency1,GFTadjacency2,Gvertex}, graph filters \cite{Graphonfilter}, graph sampling \cite{Gsamp, GFTsamp,GUncertainty,GFTSSS,GFTEfficient,Gdualizing,GFRFTsamp,GLCTsamp}, and graph signal estimation \cite{estimation}. These concepts and tools have substantially facilitated the application of GSP theory across diverse practical scenarios.

Traditional GSP theory is grounded in transformations based on orthogonal bases within finite-dimensional vector spaces \cite{GFTlaplace,GFTadjacency1,GFTadjacency2}. In this framework, graph signals are typically defined by assigning a real-valued quantity to each vertex. However, in many real-world scenarios such as sensor networks, social networks, or biological networks, graph signals are inherently time-varying. Consequently, each vertex may carry not only a real-valued signal but also temporal information \cite{HGFT,JFT}. In response to this, concepts and applications associated with joint harmonic analysis of time-varying graph signals have matured in recent years. Techniques such as the joint time-vertex Fourier transform (JFT) and associated filtering methods have established a powerful link between classical time-domain signal processing techniques and emerging GSP tools \cite{Jfilter1,Jfilter2,Jfilter3}. These developments, collectively referred to as time-vertex graph signal processing, integrate the discrete Fourier transform (DFT) in the time domain with the GFT in the vertex domain. This unified framework has been effectively deployed in various applications, including the transformation and filtering of time-varying graph signals \cite{JSpatio-temporal}, reconstruction of such signals \cite{Jsamp1,Jsamp2,Jsamp3}, prediction of joint spectral-temporal data \cite{Jprediction}, forecasting the evolution of static graph signals \cite{Jforecasting}, and semi-supervised learning \cite{Jlearning}.

Due to the combined nature of time and vertex domains, time-vertex graph signals often involve large volumes of data, making their observation, storage, and processing increasingly challenging. Consequently, sampling a subset of the data rather than recording all original information is highly beneficial, as it reduces data redundancy while preserving the essential characteristics of the signal \cite{Gsamp}. For instance, in social networks, surveying a small subset of individuals can effectively approximate the overall opinion of population, significantly saving resources. In recent years, sampling methods tailored for joint time-vertex graph signals have gained substantial attention. The spectral representation of these signals can be obtained using the JFT \cite{JFT}. For example, \cite{Jproduct} employs a product graph model to define smooth signals in the joint time-vertex domain and proposes a corresponding recovery strategy. Additionally, \cite{Jreconstruction} introduces a batch reconstruction method for time-varying graph signals based on temporal difference smoothness. Furthermore, a critical sampling set algorithm with minimal sample requirements has been developed \cite{Jsamp3}, while \cite{HGFT} extends the time domain into the Hilbert space, proposing a generalized graph signal sampling framework.

Despite capability of the JFT in sampling time-vertex graph signals, its structure is analogous to that of the DFT and GFT. Specifically, JFT performs a purely time to frequency transformation and a purely vertex to graph spectrum transform. These transformations rely on the complete signal information in both domains but lack flexibility during the conversion process. Moreover, JFT exhibits limitations in effectively handling \textit{chirp-like} features in graph signals \cite{GFRFT} and faces challenges in addressing degrees of freedom and efficiently processing non-stationary signals \cite{GLCT}. Consequently, the application of JFT remains limited across various practical domains. To address these limitations, \cite{JFRFT} introduced the joint time-vertex fractional Fourier transform (JFRFT), an extension of the JFT that integrates the discrete fractional Fourier transform (DFRFT) \cite{DFRFT} and the graph fractional Fourier transform (GFRFT) \cite{GFRFT}. The JFRFT introduces two additional parameters, enabling signal analysis within the fractional time and fractional graph spectrum domains. This enhanced flexibility, stemming from its dual-domain nature, allows for improved observation adaptability.  

Inspired by the sampling techniques for non-stationary graph signals in static graphs \cite{GLCTsamp} and the JFRFT framework, this paper presents a novel JFRFT-based sampling framework. This framework integrates the DFRFT in the time domain with the GFRFT in the vertex domain. By incorporating fractional-order operations, the GFRFT extends the GFT to model a broader range of signals. Building upon this foundation, the JFRFT unifies these capabilities into a cohesive time-vertex analysis framework, effectively capturing complex time-graph correlations through its tensor product structure \cite{JFRFT}. This structure enables the JFRFT to generalize the concept of bandlimitedness, overcoming the limitations of JFT and providing a more flexible foundation for modeling non-stationary and multi-angle signals. The proposed sampling theory establishes conditions for the perfect recovery of joint bandlimited signals and introduces an optimal sampling set selection strategy.

Furthermore, most conventional sampling methods for bandlimited signals rely on transform bases or corresponding vectors derived from eigendecomposition. In this work, within the JFRFT framework, we further investigate the connection between sampling strategies and localization operators. Specifically, we introduce a novel joint fractional localization filtering operator that not only serves as a surrogate for representing the sampling process, but also naturally facilitates the derivation of corresponding signal reconstruction formulas and error bounds. Optimal sampling operators can be equivalently expressed in terms of this localization operator. Numerical simulations and real-world experiments validate the superior performance of the proposed sampling method in terms of both recovery accuracy and computational efficiency. These findings offer new insights into effective time-varying signal processing strategies.

The main contributions are as follows.
\begin{itemize}
	\item{First, a unified JFRFT sampling theory is established by deriving perfect recovery conditions for joint bandlimited signals and introducing optimal sampling set selection strategies.}
	\item{Second, the relationship between sampling strategies and localization operators is analyzed to improve the efficiency of large-scale time-vertex signal processing.}
	\item{Finally, the proposed methods are validated through numerical experiments, demonstrating advantages in both reconstruction accuracy and computational efficiency.}
\end{itemize}

The remainder of this paper is organized as follows. Section \ref{Preliminaries} introduces the fundamental theory of graph signals, graph transforms, and sampling. Section \ref{JFRFTSampling} proposes the JFRFT-based sampling method and the optimal sampling set selection strategies. Section \ref{Relationship} investigates the relationship between sampling strategies and localization operators. Section \ref{Experiments} presents experimental results that highlight the advantages of proposed methods in recovery accuracy and computational efficiency. Finally, Section \ref{Conclusion} concludes the paper.

\section{Preliminaries}
\label{Preliminaries}
This section provides a concise overview of the GFRFT within the graph signal processing framework, along with its sampling theory and previous developments in the JFRFT.

\subsection{Graph fractional Fourier transform}
In GSP \cite{GFTlaplace,GFTadjacency1,GFTadjacency2}, a signal with a high-dimensional topological structure can be represented by a graph $\mathcal{G} = (\mathcal{V}, \mathcal{E}, \mathbf{W})$, where $\mathcal{V} = \{ v_0, \dots, v_{N-1} \}$ denotes the vertex set, and $\mathbf{W} \in \mathbb{C}^{N \times N}$ is the weighted adjacency matrix. Each entry $w_{m,n}$ specifies the weighted connection from vertex $v_m$ to vertex $v_n$. For undirected graphs, $\mathbf{W}$ is symmetric. A graph signal is defined as a map $x: \mathcal{V} \to \mathbb{C}$, and can be represented as a complex-valued vector, $\bm{x} = [x_0, x_1, \ldots, x_{N-1}]^{\top} \in \mathbb{C}^N$. 

The GFT of the graph signal $\bm{x} \in \mathbb{C}^N$ is defined as 
\begin{equation}
	\hat{\bm{x}} = \mathbf{U}^{-1} \bm{x}, \label{GFT}
\end{equation}
with its corresponding inverse GFT given by $\bm{x} = \mathbf{U} \hat{\bm{x}}$, where $\hat{\bm{x}}$ represents the frequency coefficients in the GFT domain. 

The GFT matrix $\mathbf{U}^{-1}$ is derived from the eigendecomposition of a graph shift operator $\mathbf{S}$, which is typically defined as the adjacency matrix $\mathbf{W}$, the graph Laplacian $\mathbf{L} = \mathbf{D}^{\mathrm{deg}} - \mathbf{W}$, or their normalized variants. Here, $\mathbf{D}^{\mathrm{deg}} = \mathrm{diag}(\mathbf{W} \mathbf{1})$ denotes the degree matrix. The spectral decomposition of $\mathbf{S}$ is given by 
\[
\mathbf{S} = \mathbf{U} \mathbf{\Delta} \mathbf{U}^{-1},
\] 
where the columns of $\mathbf{U}$ are the eigenvectors of $\mathbf{S}$, and the diagonal elements $\delta_0, \dots, \delta_{N-1}$ of $\mathbf{\Delta} \in \mathbb{C}^{N \times N}$ represent the graph frequencies, typically ordered from low to high.

For DFRFT in the time domain \cite{DFRFT}, it can be interpreted as a linear operator that rotates the time axis towards the frequency axis by an angle $\theta = \alpha\pi/2$, rather than the $\pi/2$ rotation in the DFT. Based on the eigenvalue decomposition method, the DFRFT is defined as
\begin{equation}
	\mathbf{F}^{\alpha}[m,n] = \sum_{k=0}^{N-1} \bm{u}_k[m] e^{j\frac{\pi}{2}k\alpha} \bm{u}_k[n],  \label{DFRFT}
\end{equation}
where $\bm{u}_k[n]$ represents the $k$-th discrete Hermite-Gaussian sequence, which serves as the eigenfunction of the DFT \cite{DFT}. The term $e^{j\frac{\pi}{2}k\alpha}$ can be interpreted as the $\alpha$-th power of eigenvalues $e^{j\frac{\pi}{2}k}$ of the DFT matrix.  

Analogous to the DFRFT, the generalized $\beta$th GFRFT \cite{GFRFT} of the signal $\bm{x}$ is defined as
\begin{equation}
	\hat{\bm{x}} = \mathbf{F}^{\beta}_{\mathcal{G}} \bm{x} = \mathbf{Q\Lambda}^{\beta}\mathbf{Q}^{-1} \bm{x}, \quad \beta \in \mathbb{R}, \label{GFRFT}
\end{equation}
with its inverse given by
\begin{equation}
	\bm{x} = \mathbf{F}^{-\beta}_{\mathcal{G}} \hat{\bm{x}} = \mathbf{Q\Lambda}^{-\beta} \mathbf{Q}^{-1} \hat{\bm{x}},\label{IGFRFT}
\end{equation}
where the matrices $\mathbf{Q}$ and $\mathbf{\Lambda}$ are obtained from the spectral decomposition of the GFT matrix,  
\[
\mathbf{F}_{\mathcal{G}} = \mathbf{U}^{-1} = \mathbf{Q\Lambda Q}^{-1}.
\]  
The fractional power $\mathbf{\Lambda}^{\beta}$ is computed by raising each diagonal entry of $\mathbf{\Lambda}$ to the power of $\beta$. Notably, the GFRFT reduces to the original signal when $\beta = 0$, and coincides with the conventional GFT when $\beta = 1$.
\begin{remark}
	To avoid notational redundancy, the orders $\alpha$ and $\beta$ are used to distinguish between the DFRFT and GFRFT, respectively. Specifically, $\mathbf{F}^\alpha$ denotes the DFRFT, while $\mathbf{F}^\beta$ refers to the GFRFT, without altering mathematical structure of the transform.
\end{remark}

\subsection{Sampling theory of graph signals}
Consider a sampling set $\mathcal{S} = \left( \mathcal{S}_0, \dots, \mathcal{S}_{M-1} \right)$, where each $\mathcal{S}_i \in \{0, 1, \dots, N-1\}$. The dimension coefficient $M$ denotes the number of sampling indices for the sampled signal $\bm{x}_{M} \in \mathbb{C}^{M}$, derived from the original signal $\bm{x} \in \mathbb{C}^N$.  

The sampling operator $\mathbf{\Psi}$ is an $M \times N$ binary matrix, defined as a linear mapping from $\mathbb{C}^N$ to $\mathbb{C}^{M}$ \cite{GFTsamp}, with
\begin{equation}
	\mathbf{\Psi}_{i,j} =
	\begin{cases}
		1, & \text{if } j = \mathcal{S}_i, \\
		0, & \text{otherwise}.
	\end{cases}  \label{Psi}
\end{equation}

The interpolation operator $\mathbf{\Phi}$ maps $\mathbb{C}^{M}$ back to $\mathbb{C}^N$. The sampling and interpolation processes are given by $\bm{x}_{M} = \mathbf{\Psi} \bm{x}$ and $\bm{x}' = \mathbf{\Phi} \bm{x}_{M} = \mathbf{\Phi} \mathbf{\Psi} \bm{x}$, where $\bm{x}'$ approximates or reconstructs $\bm{x}$.

\subsection{Joint time-vertex fractional Fourier transform}
To facilitate joint time-vertex processing, it is essential to formalize the notion of a \textit{time-vertex graph signal} \cite{JFT}. Denote $\mathbf{X} \in \mathbb{C}^{N \times T}$ as a time-vertex signal with its vectorized form $\bm{x} = \text{vec}(\mathbf{X}) \in \mathbb{C}^{NT}$. Here, $\bm{x}_t \in \mathbb{C}^N$ represents the graph signal at time $t$, and the matrix $\mathbf{X} = [\bm{x}_1, \bm{x}_2, \ldots, \bm{x}_T] \in \mathbb{C}^{N \times T}$ describes the complete time-varying graph signal. The transpose and Hermitian transpose of $\mathbf{X}$ are denoted by $\mathbf{X}^{\top}$ and $\mathbf{X}^{\mathrm{H}}$, respectively.

Given a signal $\mathbf{X}$, where its columns correspond to graph signals and its rows represent time series signals at each graph vertex, the DFT is expressed as 
\[
\text{DFT}\left\{ \mathbf{X} \right\} = \mathbf{X} \mathbf{U}_{T}^{\top},
\]
where $\mathbf{U}_{T}$ is the DFT matrix \cite{DFT}, defined element wise as $ \exp(-j2\pi tk/T)$, for $t, k = 0, 1, \ldots, T-1$. Similarly, the GFT is defined as
\[
\text{GFT}\left\{ \mathbf{X} \right\} = \mathbf{U}^{-1} \mathbf{X},
\]
where $\mathbf{U}^{-1}$ is the GFT matrix as defined in Eq. \eqref{GFT}. Consequently, the JFT is given by
\[
\text{JFT}\left\{ \mathbf{X} \right\} = \mathbf{U}^{-1} \mathbf{X} \mathbf{U}_{T}^{\top}.
\]

Since graph signals are typically vectorized, the JFT in vector form becomes
\[
\bm{\hat{x}} = \text{JFT}\left\{ \bm{x} \right\} = \mathbf{J}\bm{x} = (\mathbf{U}^{\top}_{T} \otimes \mathbf{U}^{-1}) \bm{x},
\]
where $\otimes$ represents the Kronecker product, and the relationship between the matrix and vectorized forms follows Kronecker product properties.

The JFT can be extended to the fractional domain by defining the JFRFT. For fractional orders $(\alpha, \beta) \in \mathbb{R}$, the JFRFT is expressed as  
\begin{equation}
	\mathbf{\hat{X}} = \text{JFRFT}\left\{ \mathbf{X} \right\} = \mathbf{F}^{\beta} \mathbf{X} (\mathbf{F}^{\alpha})^{\mathrm{H}},
\end{equation}
where the DFRFT is defined as Eq. \eqref{DFRFT} and $\mathbf{F}^{\alpha}$ represents the GFRFT in Eq. \eqref{GFRFT} \cite{JFRFT}. When vectorizing the time-vertex signal $\mathbf{X}$ into $\bm{x} = \text{vec}(\mathbf{X})$, the vectorized JFRFT becomes
\begin{equation}
	\hat{\bm{x}} = \text{JFRFT}\left\{ \bm{x} \right\} = \mathbf{J}^{\alpha,\beta} \bm{x} = (\mathbf{F}^{\alpha} \otimes \mathbf{F}^{\beta}) \bm{x},
\end{equation}
where, the operator $\mathbf{J}^{\alpha, \beta}(\cdot): \mathbb{C}^{NT} \rightarrow \mathbb{C}^{NT}$ denotes the JFRFT in vectorized form, and its equivalence to the matrix form arises from properties of the Kronecker product.

\section{Sampling theory of the JFRFT}
\label{JFRFTSampling}
The GFRFT sampling framework \cite{GFRFTsamp,GLCTsamp} is extended to a novel JFRFT-based approach that leverages joint orthonormal bases and spectral domain sampling to reconstruct bandlimited time-vertex signals from subsampled observations. A central step is the formalization of $\alpha,\beta$-bandlimitedness in the JFRFT domain, in line with the common emphasis on signal smoothness or bandlimitedness in graph signal sampling \cite{Gsamp}.

\subsection{$\alpha,\beta$-Bandlimited time-vertex graph signals}
\label{sec3.1}
To define bandlimited signals in the JFRFT domain, we begin by constructing a sampling framework in the GFRFT domain using fractional spectral bandlimitedness. Let $\mathcal{S}_G \subset \mathcal{V}$ represent a vertex subset. The vertex-limiting operator is $\mathbf{D}_G = \mathrm{diag}(\mathbf{1}_{\mathcal{S}_G})$, where $\mathbf{1}_{\mathcal{S}_G}$ is a binary indicator vector for $\mathcal{S}_G$ \cite{GUncertainty}. Notably, $\mathbf{D}_G = \mathbf{\Psi}^{\top} \mathbf{\Psi}$, with $\mathbf{\Psi}$ as defined in Eq.~\eqref{Psi}.

Given the inverse GFRFT matrix $\mathbf{F}^{-\beta}$ (cf. Eq.~\eqref{IGFRFT}) and a frequency index subset $\mathcal{F}_G \subset \hat{\mathcal{G}} = \{1, \dots, N\}$, the spectral-limiting operator is defined as $\mathbf{B}_G^{\beta} = \mathbf{F}^{-\beta} \mathbf{\Sigma}_{\mathcal{F}_G} \mathbf{F}^{\beta}$, where $\mathbf{\Sigma}_{\mathcal{F}_G} = \mathrm{diag}(\mathbf{1}_{\mathcal{F}_G})$. The operator $\mathbf{B}_G^{\beta}$ projects signals onto the subspace spanned by $\mathbf{F}^{-\beta}$ indexed by $\mathcal{F}_G$. Both $\mathbf{D}_G$ and $\mathbf{B}_G^{\beta}$ are Hermitian and idempotent, acting as orthogonal projectors. They satisfy
\[
\mathbf{D}_G \mathbf{X} = \mathbf{X}, \text{and} \ \ \mathbf{B}_G^{\beta} \mathbf{X} = \mathbf{X}.
\]
The complement of $\mathcal{S}_G$, denoted $\overline{\mathcal{S}}_G$, satisfies $\mathcal{V} = \mathcal{S}_G \cup \overline{\mathcal{S}}_G$, and the corresponding projection operator is $\overline{\mathbf{D}}_G = \mathbf{I} - \mathbf{D}_G$. Similarly, for $\mathcal{F}_G$, the complement set $\overline{\mathcal{F}}_G$ defines the projection $\overline{\mathbf{B}}_G^{\beta} = \mathbf{I} - \mathbf{B}_G^{\beta}$.

For the time domain, let $\mathcal{S}_T \subset \mathcal{T} = \{ t_0, \dots, t_{T-1} \}$ represent a time subset. The time-limiting operator is $\mathbf{D}_T = \mathrm{diag}(\mathbf{1}_{\mathcal{S}_T})$, and the fractional spectral-limiting operator is $\mathbf{B}_T^{\alpha} = \mathbf{F}^{-\alpha} \mathbf{\Sigma}_{\mathcal{F}_T} \mathbf{F}^{\alpha}$, where $\mathcal{F}_T \subset \hat{\mathcal{T}} = \{1, \dots, T\}$, and $\mathbf{\Sigma}_{\mathcal{F}_T} = \mathrm{diag}(\mathbf{1}_{\mathcal{F}_T})$. These project onto the sets of time domain supported and $\alpha$-bandlimited signals, respectively. The time domain constraints are
\[
\mathbf{X} \mathbf{D}_T = \mathbf{X}, \text{and} \ \  \mathbf{X} \mathbf{B}_T^{\alpha} = \mathbf{X},
\]
with complementary projections $\overline{\mathcal{S}}_T = \mathcal{T} \setminus \mathcal{S}_T$ and $\overline{\mathcal{F}}_T = \hat{\mathcal{T}} \setminus \mathcal{F}_T$, and corresponding orthogonal projections $\overline{\mathbf{D}}_T = \mathbf{I} - \mathbf{D}_T$, $\overline{\mathbf{B}}_T^{\alpha} = \mathbf{I} - \mathbf{B}_T^{\alpha}$.

This unified operator-based approach defines jointly bandlimited signals in the JFRFT domain. A signal $\mathbf{X} \in \mathbb{C}^{N \times T}$ is jointly supported on $\mathcal{S}_G \times \mathcal{S}_T$ and bandlimited over $\mathcal{F}_G \times \mathcal{F}_T$ if and only if
\[
\mathbf{D}_G \mathbf{X} \mathbf{D}_T = \mathbf{X}, \text{and} \ \  \mathbf{B}_G^{\beta} \mathbf{X} \mathbf{B}_T^{\alpha} = \mathbf{X}.
\]

Let $\bm{x} = \mathrm{vec}(\mathbf{X}) \in \mathbb{C}^{NT}$ denote the vectorized form of the time-vertex signal. We define the joint support and bandlimiting projectors as $\mathbf{D}_J = \mathbf{D}_T \otimes \mathbf{D}_G$ and $\mathbf{B}_J^{\alpha, \beta} = \mathbf{B}_T^{\alpha} \otimes \mathbf{B}_G^{\beta}$. The joint localization conditions are expressed as
\begin{equation}
	\mathbf{D}_J \bm{x} = \bm{x}, \quad \mathbf{B}_J^{\alpha, \beta} \bm{x} = \bm{x}.
\end{equation}
By leveraging the Kronecker product, the joint bandlimiting operator is
\[
\mathbf{B}_J^{\alpha, \beta} = \mathbf{J}^{-\alpha, -\beta} \mathbf{\Sigma}_{\mathcal{F}_J} \mathbf{J}^{\alpha, \beta},
\]
where $\mathcal{F}_J = \mathcal{F}_T \times \mathcal{F}_G$ is the joint spectral support set. Thus, we arrive at the joint perfect bandlimited theorem in the JFRFT domain.
\begin{thm} \label{bandlimited}
	Let $\bm{x} \in \mathbb{C}^{NT}$ represent a time-vertex signal. Define the joint support and frequency projection operators $\mathbf{D}_J$ and $\mathbf{B}_J^{\alpha, \beta}$. Then, $\bm{x}$ is perfectly localized over $\mathcal{S}_G \times \mathcal{S}_T$ and $\mathcal{F}_G \times \mathcal{F}_T$ if and only if,
	$
	\lambda_{\max} \left( \mathbf{B}_J^{\alpha, \beta} \mathbf{D}_J \mathbf{B}_J^{\alpha, \beta} \right) = 1.
	$
	In this case, $\bm{x}$ is an eigenvector associated with the unit eigenvalue. Equivalently, this condition is characterized by
	$
	\left\| \mathbf{B}_J^{\alpha, \beta} \mathbf{D}_J \right\|_2 = \left\| \mathbf{D}_J \mathbf{B}_J^{\alpha, \beta} \right\|_2 = 1.
	$
\end{thm}
\begin{proof}
	The proof is reported in \ref{AA}.
\end{proof}

\subsection{Joint fractional sampling and perfect recovery}
\label{sec3.2}
For the time-vertex graph signal $\mathbf{X} \in \mathbb{C}^{N \times T}$, the sampled signal is given by
\[
\mathbf{X}_{\mathcal{S}_G \times \mathcal{S}_T} = \mathbf{D}_G \mathbf{X} \mathbf{D}_T,
\]
and the sampling can also be expressed in vector form. For the time-vertex graph signal vector $\bm{x} \in \mathbb{C}^{NT}$ and the joint sampling matrix operator $\mathbf{D}_J \in \mathbb{C}^{NT \times NT}$, the sampled signal is represented as
\begin{equation}
\bm{x}_{\mathcal{S}_J} = \mathbf{D}_J \bm{x}, \label{DJ}
\end{equation}
where $\mathcal{S}_J$ denotes the sample set of the time-vertex signal, and $\bm{x}_{\mathcal{S}_J}$ represents the subvector of $\bm{x}$ obtained by extracting the elements indexed by $\mathcal{S}_J$.

We can perfectly recover $\bm{x}$ from $\bm{x}_{\mathcal{S}_J}$ using the recovery operator $\mathbf{R}_J$. Next, we define the concept of bandwidth.
\begin{defn}
\label{defbandwidth}
A signal $\bm{x} \in \mathbb{C}^{NT}$ is said to be bandwidth-limited if its joint spectral representation $\hat{\bm{x}}$ has fewer than $K_J < NT$ non-zero entries. The time bandwidth $K_T$ is defined as the number of non-zero rows in $\hat{\mathbf{X}}$, and the vertex bandwidth $K_G$ as the number of non-zero columns, capturing the support over the time and graph domains, respectively.
\end{defn}

If the projection bandwidth satisfies $K_T < T$ and $K_G < N$, the signal $\mathbf{X}$ is considered jointly bandlimited. It is evident that the general relationship is given by
\[
\max(K_T, K_G) \leq K_J \leq K_T K_G.
\]

Let $\mathcal{S}_G \subseteq \mathcal{V}$ and $\mathcal{S}_T \subseteq \mathcal{T}$ be appropriate subsets. There exist suitable sampling set sizes $M_T \geq K_T$ and $M_G \geq K_G$, such that the signal $\mathbf{X}$ can be perfectly recovered from the sampled signal. The reconstructed signal is given by
\[
\mathbf{X}'= \mathbf{R}_G \mathbf{X}_{\mathcal{S}_G \times \mathcal{S}_T} \mathbf{R}_T^{\mathrm{H}} = \mathbf{R}_G \mathbf{D}_G \mathbf{X} \mathbf{D}_T \mathbf{R}_T^{\mathrm{H}}.
\]
Thus, using the recovery operators $\mathbf{R}_G$ and $\mathbf{R}_T$, $\mathbf{X}$ can be perfectly recovered from the sampled signal $\mathbf{X}_{\mathcal{S}_G \times \mathcal{S}_T}$.

The actual sampling set can be expressed as $\mathcal{S}_J = \mathcal{S}_T \times \mathcal{S}_G$, with a total number of samples $M_J = M_T M_G$. The vectorized form of the recovered signal is given by
\begin{equation}
\bm{x}' = \mathbf{R}_J\bm{x}_{\mathcal{S}_J}=\mathbf{R}_J\mathbf{D}_J \bm{x}. \label{x=RDx}
\end{equation}
Therefore, the recovery operator for the set $\mathcal{S}_J$ is $\mathbf{R}_J = \mathbf{R}_T \otimes \mathbf{R}_G$. The key problem is to determine the conditions under which $\bm{x}$ can be perfectly recovered from the sampled signal $\bm{x}_{\mathcal{S}_J}$. The following theorem addresses this, building on sampling principles from \cite{GFTsamp, GFRFTsamp, GLCTsamp}.
\begin{thm}
For the sampling operator $\mathbf{D}_J \in \mathbb{C}^{NT \times NT}$, the recovery operator $\mathbf{R}_J \in \mathbb{C}^{NT \times NT}$ spans the space $\text{BL}_{K_J}(\mathbf{J}^{\alpha, \beta})$, meaning $\bm{x}$ is $\alpha, \beta$-bandlimited. Furthermore, if $\mathbf{R}_J \mathbf{D}_J$ functions as a projection operator, perfect recovery is achieved as $\bm{x} = \mathbf{R}_J \mathbf{D}_J \bm{x} = \mathbf{R}_J \bm{x}_{\mathcal{S}_J}$ for all $\bm{x} \in \text{BL}_{K_J}(\mathbf{J}^{\alpha, \beta})$.
\end{thm}
\begin{proof}
The proof is reported in \ref{AB}.
\end{proof}

Thus, if the original signal $\bm{x}$ can be perfectly recovered, we have the following
\[
\bm{x} - \mathbf{R}_J \mathbf{D}_J \bm{x} = \bm{x} - \mathbf{R}_J \left( \mathbf{I} - \overline{\mathbf{D}}_J \right) \bm{x} = \bm{x} - \mathbf{R}_J \left( \mathbf{I} - \overline{\mathbf{D}}_J \mathbf{B}^{\alpha, \beta}_J \right) \bm{x}.
\]

The perfect recovery of $\bm{x}$ implies the existence of the matrix $\mathbf{R}_J$, which is equivalent to the invertibility of $\left( \mathbf{I} - \overline{\mathbf{D}}_J \mathbf{B}^{\alpha, \beta}_J \right)$. Hence, we conclude that, 
\[
\mathbf{R}_J = \left( \mathbf{I} - \overline{\mathbf{D}}_J \mathbf{B}^{\alpha, \beta}_J \right)^{-1}.
\]

Therefore, perfect recovery of the original signal is possible when $\left\| \overline{\mathbf{D}}_J \mathbf{B}^{\alpha, \beta}_J \right\| < 1$. If the maximum norm of $\overline{\mathbf{D}}_J \mathbf{B}^{\alpha, \beta}_J$ equals 1, perfect localization on the complement of the sampling set exists, rendering signal recovery unattainable.

Furthermore, since $\bm{x}$ is an $\alpha, \beta$-bandlimited time-vertex graph signal, we have
\[
\left( \mathbf{I} - \overline{\mathbf{D}}_J \mathbf{B}^{\alpha, \beta}_J \right) \bm{x} = \left( \mathbf{I} - \overline{\mathbf{D}}_J \right) \bm{x} = \mathbf{D}_J \bm{x} = \mathbf{D}_J \mathbf{B}^{\alpha, \beta}_J \bm{x}.
\]

By analyzing the preceding equation, it is evident that $\left( \mathbf{I} - \overline{\mathbf{D}}_J \mathbf{B}^{\alpha, \beta}_J \right)$ is equivalent to $\mathbf{D}_J \mathbf{B}^{\alpha, \beta}_J$, leading to the conclusion that $\mathbf{R}_J = \left( \mathbf{D}_J \mathbf{B}^{\alpha, \beta}_J \right)^{-1}$. However, for a degenerate matrix, where $\mathbf{D}_J \mathbf{B}^{\alpha, \beta}_J$ may not be full rank, we define $\mathbf{R}_J = \left( \mathbf{D}_J \mathbf{B}^{\alpha, \beta}_J \right)^{\dag}$ using the pseudo-inverse. Consequently, the recovery operator can be rewritten as
\begin{equation}
\mathbf{R}_J = \left( \mathbf{D}_J \mathbf{B}^{\alpha,\beta}_J \right)^{\dag} = \left( \mathbf{D}_J \mathbf{J}^{-\alpha,-\beta} \mathbf{\Sigma}_{\mathcal{F}_J} \mathbf{J}^{\alpha,\beta} \right)^{\dag} = \mathbf{J}^{-\alpha,-\beta}_{\mathcal{J} \mathcal{F}_J} \left(\mathbf{J}^{-\alpha,-\beta}_{\mathcal{S}_J\mathcal{F}_J} \right)^{\dag}, \label{RJ}
\end{equation}
where, $\mathbf{J}^{-\alpha,-\beta}_{\mathcal{J} \mathcal{F}_J}$ denotes the submatrix of $\mathbf{J}^{-\alpha,-\beta}$ formed by selecting all rows (i.e., $\mathcal{J} = [NT]$) and columns indexed by the joint frequency support $\mathcal{F}_J$, while $\mathbf{J}^{-\alpha,-\beta}_{\mathcal{S}_J \mathcal{F}_J}$ selects rows indexed by the sampling set $\mathcal{S}_J$ and the same columns indexed by $\mathcal{F}_J$.

\subsection{Optimal joint fractional sampling operator}
Effective sampling of graph signals requires not only determining the number of samples but also selecting optimal sampling nodes, as their locations critically affect reconstruction performance. This holds for time-vertex signals as well, where our strategy follows two key principles: minimizing noise sensitivity and maximizing signal localization—aligned with prior works such as \cite{GUncertainty, GFTsamp, GFTEfficient, GFTSSS, Gdualizing, GFRFTsamp, GLCTsamp}.

\subsubsection{JFRFT sampling based on error minimization}
When the joint bandwidth $K_J$ does not exceed the number of joint samples $M_J$, one can formulate sampling set selection as an optimization problem to minimize reconstruction error. For a noisy observation modeled by $\bm{x}' = \mathbf{R}_J (\mathbf{D}_J \bm{x} + \bm{e})$, the resulting error becomes
$$
\bm{\epsilon} = \bm{x}' - \bm{x} = \mathbf{R}_J \bm{e} = \mathbf{J}^{-\alpha,-\beta}_{\mathcal{J} \mathcal{F}_J} \left(\mathbf{J}^{-\alpha,-\beta}_{\mathcal{S}_J\mathcal{F}_J} \right)^{\dag} \bm{e},
$$

\textit{Maximizing singular value of the minimum (MaxSigMin):} Minimizing the \(\ell_2\)-norm of the error yields the following optimization problem,
\[
\mathcal{S}_J^{\text{opt}} = \argmin_{\mathcal{S}_J \subseteq \mathcal{J}} \left\| \bm{\epsilon} \right\|_2=\arg \min_{\mathcal{S}_J \subseteq \mathcal{J}} \left\| \mathbf{J}^{-\alpha,-\beta}_{\mathcal{J} \mathcal{F}_J} \left( \mathbf{J}^{-\alpha,-\beta}_{\mathcal{S}_J \mathcal{F}_J} \right)^{\dagger} \bm{e} \right\|_2.
\]
Using the Cauchy–Schwarz inequality, the upper bound of the error norm can be expressed as
\[
\left\| \bm{\epsilon} \right\|_2 \leq \left\| \mathbf{J}^{-\alpha,-\beta}_{\mathcal{J} \mathcal{F}_J} \right\|_2 \cdot \left\| \left( \mathbf{J}^{-\alpha,-\beta}_{\mathcal{S}_J \mathcal{F}_J} \right)^{\dagger} \right\|_2 \cdot \left\| \bm{e} \right\|_2.
\]
Since $\left\| \mathbf{J}^{-\alpha,-\beta}_{\mathcal{J} \mathcal{F}_J} \right\|_2$ and $\left\| \bm{e} \right\|_2$ are constant with respect to $\mathcal{S}_J$, the objective reduces to minimizing the spectral norm of the pseudo-inverse. This is equivalent to maximizing the smallest singular value of the matrix $\mathbf{J}^{-\alpha,-\beta}_{\mathcal{S}_J \mathcal{F}_J}$, namely
\begin{equation}
\mathcal{S}_J^{\text{opt}} = \argmax_{\mathcal{S}_J \subseteq \mathcal{J}}~ \sigma_{\min} \left( \mathbf{J}^{-\alpha,-\beta}_{\mathcal{S}_J \mathcal{F}_J} \right). \label{SJ1}
\end{equation}

This optimization problem can be efficiently solved using a greedy strategy. At the \(m\)-th iteration, the next vertex \(y^{\text{opt}}\) is chosen by
\[
y^{\mathrm{opt}} = \argmax_{y \in \mathcal{J} \setminus \mathcal{S}_{J_m}}~\sigma_{\min}  \left( \mathbf{J}^{-\alpha,-\beta}_{(\mathcal{S}_{J_m} \cup y) \mathcal{F}_J} \right). 
\]

\textit{Minimizing trace (MinTrac):} Aims to minimize the trace of the error covariance matrix $\mathbf{E}$. This corresponds to minimizing the total energy of the reconstruction error. Formally, the optimization problem is given by
\begin{equation}
\begin{aligned}
	\mathcal{S}_J^{\text{opt}} 
	&= \argmin_{\mathcal{S}_J \subseteq \mathcal{J}} \ \mathrm{tr}(\mathbf{E}) = \argmin_{\mathcal{S}_J \subseteq \mathcal{J}} \ \mathrm{tr}(\bm{\epsilon} \bm{\epsilon}^*) \\
	&= \argmin_{\mathcal{S}_J \subseteq \mathcal{J}} \ \mathrm{tr}\left[ \mathbf{J}^{-\alpha,-\beta}_{\mathcal{J}\mathcal{F}_J} \left( \mathbf{J}^{-\alpha,-\beta}_{\mathcal{S}_J\mathcal{F}_J} \mathbf{J}^{\alpha,\beta}_{\mathcal{S}_J\mathcal{F}_J} \right)^{-1} \mathbf{J}^{\alpha,\beta}_{\mathcal{J}\mathcal{F}_J} \right] \\
	&= \argmin_{\mathcal{S}_J \subseteq \mathcal{J}} \ \mathrm{tr} \left[ \left( \mathbf{J}^{\alpha,\beta}_{\mathcal{S}_J\mathcal{F}_J} \mathbf{J}^{-\alpha,-\beta}_{\mathcal{S}_J\mathcal{F}_J} \right)^{-1} \right].
\end{aligned} \label{SJ2}
\end{equation}
That is, the optimal sampling set minimizes the trace of the inverse of the restricted joint spectral localization operator.

As in the MaxSigMin criterion, a greedy algorithm is used to solve this optimization. At the \(m\)th iteration, the next vertex \(y^{\text{opt}}\) is selected by
\[
y^{\text{opt}} = \argmin_{y \in \mathcal{J} \setminus \mathcal{S}_{J_m}} \ \mathrm{tr} \left[ 
\left( \mathbf{J}^{\alpha,\beta}_{(\mathcal{S}_{J_m} \cup y)\mathcal{F}_J} \mathbf{J}^{-\alpha,-\beta}_{(\mathcal{S}_{J_m} \cup y)\mathcal{F}_J} \right)^{-1} \right].
\]

\subsubsection{JFRFT sampling based on localized basis}
Similarly, when $K_J \leq M_J$, the signal $\bm{x}$ can be reconstructed using the time-vertex localization operator $\mathbf{D}_J$ and the spectral projector $\mathbf{B}^{\alpha,\beta}_J$. Perfect recovery is ensured if $\bm{x}$ satisfies Theorem \ref{bandlimited}. This class of methods typically employs three representative objective functions, with reconstruction error expressed as
\begin{equation}
\bm{\epsilon} =\mathbf{R}_J \bm{e}=\left( \mathbf{D}_J \mathbf{B}^{\alpha,\beta}_J \right)^{\dag} \bm{e}= \left(\mathbf{B}^{\alpha,\beta}_J \mathbf{D}_J \mathbf{B}^{\alpha,\beta}_J\right)^{\dag} \mathbf{D}_J \bm{e}. \label{epsilon}
\end{equation}

\textit{Minimizing Frobenius norm of the pseudo-inverse (MinPinv):} Selects the sampling set that yields a recovery operator with minimal Frobenius norm,
\[
\mathcal{S}^{\mathrm{opt}}_J = \argmin_{\mathcal{S}_J \subseteq \mathcal{J}} \left\| \left(\mathbf{B}^{\alpha,\beta}_J \mathbf{D}_J \mathbf{B}^{\alpha,\beta}_J\right)^{\dag} \mathbf{D}_J \bm{e} \right\|_F.
\]
By applying the Cauchy–Schwarz inequality, the above can be upper bounded as
\[
\left\| \left(\mathbf{B}^{\alpha,\beta}_J \mathbf{D}_J \mathbf{B}^{\alpha,\beta}_J\right)^{\dag} \mathbf{D}_J \bm{e} \right\|_F \leq \left\| \left( \boldsymbol{\Sigma}_{\mathcal{F}_J} \mathbf{J}^{\alpha,\beta} \mathbf{D}_{\mathcal{S}_J} \right)^{\dag} \right\|_F \cdot \left\| \mathbf{D}_{\mathcal{S}_J} \bm{e} \right\|_F,
\]
where $\mathbf{D}_{\mathcal{S}_J} = \mathbf{D}_J$ is introduced purely for notational consistency. It denotes the submatrix of the time-vertex domain localization operator corresponding to the sampling index set $\mathcal{S}_J$. Since $\left\| \mathbf{D}_{\mathcal{S}_J} \bm{e} \right\|_F$ is constant, the optimization reduces to
\begin{equation}
\mathcal{S}^{\mathrm{opt}}_J = \argmin_{\mathcal{S}_J \subseteq \mathcal{J}} \left\| \left( \boldsymbol{\Sigma}_{\mathcal{F}_J} \mathbf{J}^{\alpha,\beta} \mathbf{D}_{\mathcal{S}_J} \right)^{\dag} \right\|_F = \argmin_{\mathcal{S}_J \subseteq \mathcal{J}} \sum_{i=1}^{K_J} \frac{1}{\sigma_i\left(\boldsymbol{\Sigma}_{\mathcal{F}_J} \mathbf{J}^{\alpha,\beta} \mathbf{D}_{\mathcal{S}_J}\right) }. \label{SJ3}
\end{equation}

This objective is addressed using the same greedy selection strategy as outlined in the preceding methods
\[
y^{\mathrm{opt}} = \argmin_{y \in \mathcal{J} \setminus \mathcal{S}_{J_m}} \sum_{i=1}^{K_J} \frac{1}{\sigma_i\left(\boldsymbol{\Sigma}_{\mathcal{F}_J} \mathbf{J}^{\alpha,\beta} \mathbf{D}_{(\mathcal{S}_{J_m} \cup y)}\right) }. 
\]

\textit{Maximizing singular value (MaxSig):} Seeks to maximize the Frobenius norm of $\mathbf{B}^{\alpha,\beta}_J \mathbf{D}_J \mathbf{B}^{\alpha,\beta}_J$, thereby promoting numerical stability,
\begin{equation}
\mathcal{S}^{\mathrm{opt}}_J = \argmax_{\mathcal{S}_J \subseteq \mathcal{J}} \left\| \boldsymbol{\Sigma}_{\mathcal{F}_J} \mathbf{J}^{\alpha,\beta} \mathbf{D}_{\mathcal{S}_J} \right\|_F = \argmax_{\mathcal{S}_J \subseteq \mathcal{J}} \sum_{i=1}^{K_J} \sigma_i\left(\boldsymbol{\Sigma}_{\mathcal{F}_J} \mathbf{J}^{\alpha,\beta} \mathbf{D}_{\mathcal{S}_J}\right) .  \label{SJ4}
\end{equation}

A greedy selection procedure is again utilized, with the optimal vertex chosen at each step according to
\[
y^{\mathrm{opt}} = \argmin_{y \in \mathcal{J} \setminus \mathcal{S}_{J_m}} \sum_{i=1}^{K_J} \sigma_i\left(\boldsymbol{\Sigma}_{\mathcal{F}_J} \mathbf{J}^{\alpha,\beta} \mathbf{D}_{(\mathcal{S}_{J_m} \cup y)}\right).
\]

\textit{Maximizing the volume of the parallelepiped (MaxVol):} Focuses on maximizing the volume formed by the selected rows of the transformation matrix. This is achieved by maximizing the determinant
\[
\mathcal{S}^{\mathrm{opt}}_J = \argmax_{\mathcal{S}_J \subseteq \mathcal{J}}~\det \left[ \mathbf{J}^{-\alpha, -\beta}_{\mathcal{S}_J\mathcal{F}_J} \mathbf{J}^{\alpha,\beta}_{\mathcal{S}_J\mathcal{F}_J} \right],
\]
which can also be equivalently expressed as
\begin{equation}
\mathcal{S}^{\mathrm{opt}}_J = \argmax_{\mathcal{S}_J \subseteq \mathcal{J}}~\det \left[\left(  \mathbf{J}^{-\alpha,-\beta} \boldsymbol{\Sigma}_{\mathcal{F}_J} \mathbf{J}^{\alpha,\beta}\right) _{\mathcal{S}_J} \right] .  \label{SJ5}
\end{equation}

Similarly, a greedy algorithm is adopted, where the vertex selected at the \( m \)-th iteration satisfies
\[
y^{\mathrm{opt}} = \argmax_{y \in \mathcal{J} \setminus \mathcal{S}_{J_m}}~\det \left[ \left( \mathbf{J}^{-\alpha,-\beta} \boldsymbol{\Sigma}_{\mathcal{F}_J} \mathbf{J}^{\alpha,\beta}  \right)_{(\mathcal{S}_{J_m} \cup y)} \right].
\]

Precise localization in both the joint time-vertex and spectral domains concentrates signal energy on specific nodes. Maximizing this energy over the sampling set $\mathcal{S}_J$ ensures that key local structures and dynamic patterns are effectively captured. For the full domain $\mathcal{J}$, a well-chosen $\mathcal{S}_J$ preserves the essential features of the time-vertex signal. This sampling exhibits a submodular property: as the set grows, the marginal gain from adding new elements diminishes. To solve this subset selection problem, we employ a greedy algorithm to construct an effective sampling operator.

To evaluate the five proposed sampling strategies, we present a case study using time-vertex signals of seasonal temperature and precipitation across U.S. states\footnote{Available: https://www.currentresults.com/Weather/US/weather-averages-index.php}. The graph $\mathcal{G}$ represents 48 contiguous U.S. states, where each node denotes a state \cite{USA,HGFRFT} and edges indicate direct geographic adjacency with unit weights. Alaska and Hawaii are excluded due to their geographic isolation. The graph topology is shown in Fig. \ref{fig01}.
\begin{figure}[h]
\begin{center}
	\begin{minipage}[t]{1\linewidth}
		\centering
		\includegraphics[width=\linewidth]{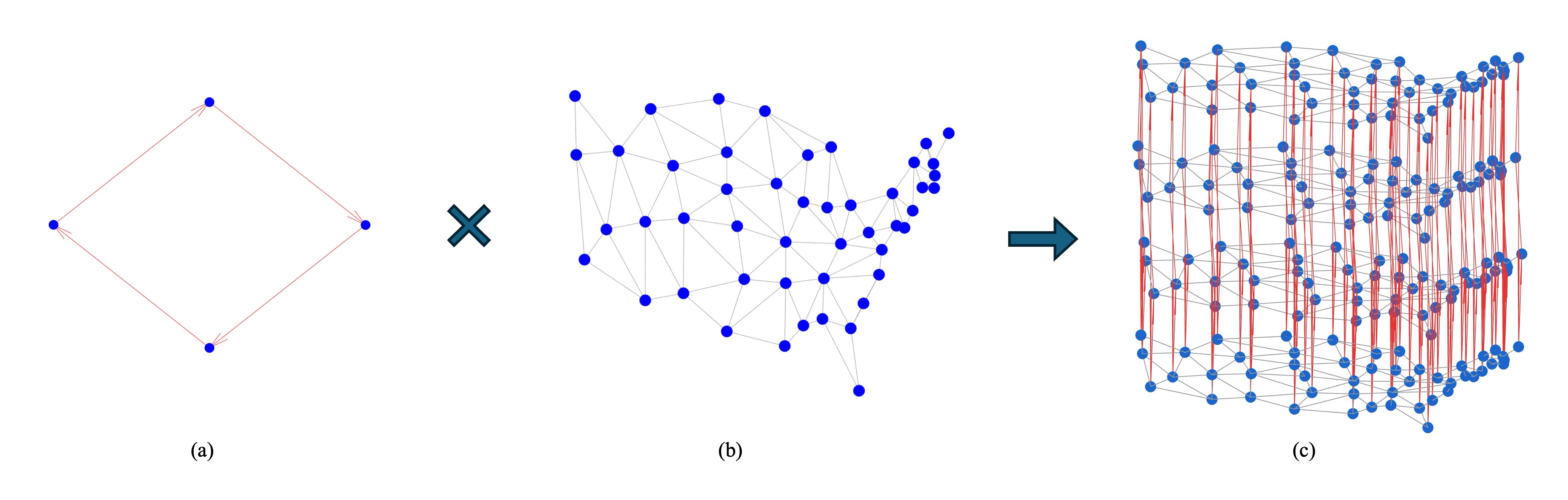}
	\end{minipage}
\end{center}
\vspace*{-20pt}
\caption{Illustration of the joint time-vertex graph constructed via a product graph: (a) directed cyclic time graph over four seasons; (b) underlying 48-node spatial graph of U.S. states; (c) the resulting 192-node joint time-vertex product graph.}
\vspace*{-3pt}
\label{fig01}
\end{figure}

We compare the five proposed JFRFT-based sampling strategies with uniform random sampling. Fig. \ref{fig02} reports the normalized mean squared error (NMSE), computed as the average reconstruction error per node normalized by the original signal magnitude. The time, vertex, and joint bandwidths are set as $K_T = 2$, $K_G = 10$, and $K_J = 20$, respectively. The number of samples $M_J$ increases from 0 to 180, and additive Gaussian noise $\bm{e} \sim \mathcal{N}(0, 0.01^2)$ is added to the sampled signal.
\begin{figure}[h]
\begin{center}
	\begin{minipage}[t]{0.45\linewidth}
		\centering
		\includegraphics[width=\linewidth]{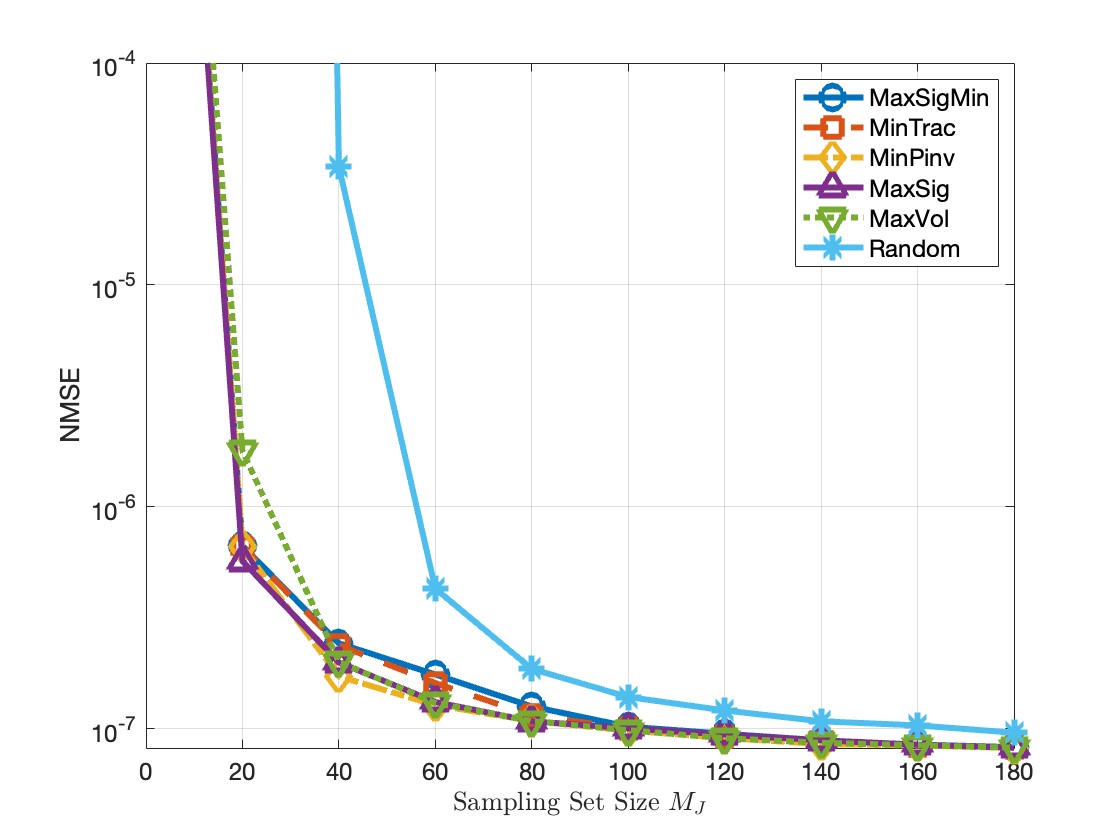}
		\parbox{2cm}{\tiny (a) Temperature signal.}
	\end{minipage}
	\begin{minipage}[t]{0.45\linewidth}
		\centering
		\includegraphics[width=\linewidth]{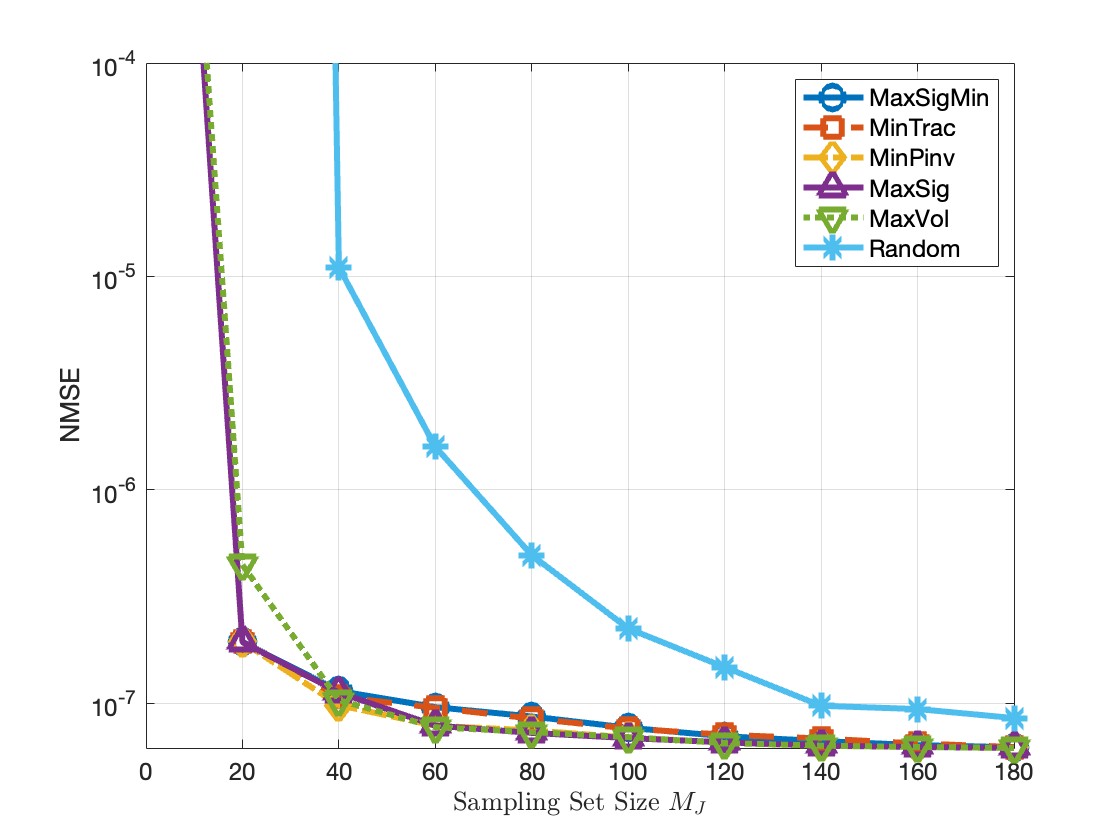}
		\parbox{2cm}{\tiny (b) Precipitation signal.}
	\end{minipage}
\end{center}
\caption{NMSE vs. number of samples for different sampling strategies.}
\vspace*{-3pt}
\label{fig02}
\end{figure}
Panel (a) presents a temperature signal bandlimited under JFRFT with parameters $\alpha = 1$, $\beta = 0.9$, while panel (b) shows a precipitation signal with $\alpha = 0.7$, $\beta = 0.5$. In both cases, we assess the NMSE performance of six sampling schemes as the number of samples increases. Results show that, except for random sampling, all methods achieve accurate recovery with relatively few samples, and reconstruction improves steadily with more data. Based on these findings, we next explore the connection between the five sampling operators and the localized filtering operator.

\section{Relationship between sampling and localization operators}
\label{Relationship}
In joint time-vertex signal processing, characterizing the relationship between sampling operators and the JFRFT domain is essential for principled sampling. We introduce a joint fractional localized filtering operator that captures the intrinsic coupling between time and graph structures \cite{GUncertainty,GFTEfficient,GFTSSS}. This operator unifies signal reconstruction and error bounds, and reformulates five representative sampling strategies in a localized form, linking geometric intuition with spectral analysis.

\subsection{Joint fractional localized filtering operator}
Consider a polynomial filter defined over the joint time-vertex domain. Let $h(\delta_i(\mathbf{S}_T))$ denote the spectral kernel in the time domain, where $\delta_i(\mathbf{S}_T)$ is the $i$-th eigenvalue of the time shift operator $\mathbf{S}_T$. Similarly, $h(\delta_i(\mathbf{S}_G))$ represents the spectral kernel in the vertex domain, ensuring structural consistency across both domains. The filtering operations for DFRFT and GFRFT are expressed as
\[
\mathbf{X} = \mathbf{Y} \mathbf{F}^{-\alpha} h(\mathbf{\Delta}_T) \mathbf{F}^{\alpha}, \quad
\mathbf{X} = \mathbf{F}^{-\beta} h(\mathbf{\Delta}_G) \mathbf{F}^{\beta} \mathbf{Y},
\]
where $\mathbf{X}$ and $\mathbf{Y}$ are the output and input signals, respectively. The diagonal matrices $\mathbf{\Delta}_T = \text{diag}(\delta_0(\mathbf{S}_T), \dots, \delta_{T-1}(\mathbf{S}_T))$ and $\mathbf{\Delta}_G = \text{diag}(\delta_0(\mathbf{S}_G), \dots, \delta_{N-1}(\mathbf{S}_G))$ collect the spectral components for the temporal and graph domains.

Combining these domains, the joint time-vertex filtering operator is given by
\[
\left( \mathbf{F}^{-\alpha} h(\mathbf{\Delta}_T) \mathbf{F}^{\alpha} \right) \otimes \left( \mathbf{F}^{-\beta} h(\mathbf{\Delta}_G) \mathbf{F}^{\beta} \right) = \mathbf{J}^{-\alpha ,-\beta} h(\mathbf{\Delta}_J) \mathbf{J}^{\alpha ,\beta},
\]
where $\mathbf{\Delta}_J$ is the diagonal joint spectral matrix, and $\mathbf{J}^{\alpha,\beta}$ is the JFRFT matrix. Following \cite{Glocal}, the localized filtering operator is defined as
\[
T^{\alpha,\beta}_{h,i}(n) = \sqrt{NT} \sum_{\ell=0}^{NT-1} h(\delta_\ell) \mathrm{j}^*_\ell(i) \mathrm{j}_\ell(n),
\]
where $n$ indexes the signal entries and $i$ denotes the center node of localization. The function $h(\delta_\ell)$ typically represents a low-pass filter, and $\mathrm{j}_\ell$ is the $\ell$-th column of the basis matrix $\mathbf{J}^{\alpha,\beta}$. In matrix form, the localized filtering operator is
\begin{equation}
\mathbf{T}^{\alpha,\beta} = \sqrt{NT}\, \mathbf{J}^{-\alpha,-\beta} h(\mathbf{\Delta}_J) \mathbf{J}^{\alpha,\beta}.
\end{equation}

Reconstruction of the signal is achieved through a weighted linear combination of the localized filtering operator $\mathbf{T}^{\alpha,\beta}$ after sampling.

\subsection{Recovery via the JFRFT localized filtering operator}
The theoretical guarantee of this reconstruction scheme is established in the following theorem.
\begin{thm}\label{thm3}
Let $\bm{x}$ be a time-vertex graph signal that is $\alpha, \beta$-bandlimited with joint spectral bandwidth $K_J$. If $M_J \geq K_J$ and the sampled signal $\bm{x}_{\mathcal{S}_J}$ from the index set $\mathcal{S}_J$ allows perfect recovery, then the original signal can be exactly reconstructed as
\begin{equation}
	\bm{x}' = \mathbf{T}^{\alpha,\beta}_{\mathcal{J}\mathcal{S}_J} \left( \mathbf{T}^{\alpha,\beta}_{\mathcal{S}_J} \right)^{\dag} \bm{x}_{\mathcal{S}_J}, \label{xR=TT}
\end{equation}
where $\mathcal{J}$ is the full set of time-vertex indices and $\mathcal{S}_J \subseteq \mathcal{J}$ is the sampling set.
\end{thm}

\begin{proof}
The proof is provided in Appendix~\ref{AC}.
\end{proof}

Based on Eqs. \eqref{x=RDx}, \eqref{RJ}, and \eqref{xR=TT}, it follows that the recovery operator $\mathbf{R}_J$ can be approximated as
\begin{equation}
\mathbf{R}_J \approx \mathbf{T}^{\alpha,\beta}_{\mathcal{J}\mathcal{S}_J} \left( \mathbf{T}^{\alpha,\beta}_{\mathcal{S}_J} \right)^{\dag}.
\end{equation}

When the sampled signal is contaminated with Gaussian noise $\bm{e}$ (as in Eq. \eqref{epsilon}), the reconstruction error must be re-evaluated using Eq. \eqref{xR=TT}. In this case, the error $\bm{\epsilon}$ is given by
\begin{equation}
\bm{\epsilon} = \bm{x}' - \bm{x} = \mathbf{R}_J \bm{e} = \mathbf{T}^{\alpha,\beta}_{\mathcal{J}\mathcal{S}_J} \left( \mathbf{T}^{\alpha,\beta}_{\mathcal{S}_J} \right)^{\dag} \bm{e}.
\end{equation}
Simplifying further
\[
\bm{\epsilon} = \mathbf{J}^{-\alpha,-\beta} h\left( \mathbf{\Delta}_J \right) \mathbf{J}^{\alpha,\beta}_{\mathcal{S}_J\mathcal{J}} \left( \mathbf{J}^{-\alpha,-\beta}_{\mathcal{S}_J\mathcal{J}} h\left( \mathbf{\Delta}_J \right) \mathbf{J}^{\alpha,\beta}_{\mathcal{S}_J\mathcal{J}} \right)^{\dag} \bm{e},
\]
\[
= \left( \mathbf{T}^{\alpha,\beta} \right)^{1/2} \left( \left( \mathbf{T}^{\alpha,\beta} \right)^{1/2}_{\mathcal{S}_J\mathcal{J}} \right)^{\dag} \bm{e}.
\]
The corresponding error covariance matrix can then be computed as
\[
\mathbf{E} = \left( \mathbf{T}^{\alpha,\beta} \right)^{1/2} \left( \left( \mathbf{T}^{\alpha,\beta} \right)^{1/2}_{\mathcal{S}_J\mathcal{J}} \right)^{\dag} \mathbf{I}_{\mathcal{S}_J} \left( \left( \mathbf{T}^{\alpha,\beta} \right)^{-1/2}_{\mathcal{S}_J\mathcal{J}} \right)^{\dag} \left( \mathbf{T}^{\alpha,\beta} \right)^{1/2}.
\]
This simplifies to
\[
\mathbf{E} = \left( \mathbf{T}^{\alpha,\beta} \right)^{1/2} \left( \left( \mathbf{T}^{\alpha,\beta} \right)^{-1/2}_{\mathcal{S}_J\mathcal{J}} \left( \mathbf{T}^{\alpha,\beta} \right)^{1/2}_{\mathcal{S}_J\mathcal{J}} \right)^{\dag} \left( \mathbf{T}^{\alpha,\beta} \right)^{1/2}.
\]
Since the reconstruction operator based on filtering is equivalent to the one derived from localization (Eq. \eqref{xR=TT}), the resulting error covariance matrix $\mathbf{E}$ is the same. To minimize the reconstruction error, the sampling strategy can be optimized by minimizing or maximizing matrix-based criteria such as the trace, determinant, or maximum eigenvalue \cite{GLCTsamp}, as summarized in Table~\ref{tab01}.
\begin{table}[t]
\caption{Minimization Methods of Error}
\label{tab01}
\footnotesize
\centering
\begin{tabular}{cl}
	\toprule
	Optimal Methods&Objective Functions\\
	\midrule
	A-optimal & $\min_{\mathcal{S}_J\subseteq \mathcal{J}} \mathrm{tr} \left[ \left( \mathbf{T}^{\alpha,\beta}_{\mathcal{S}_J}\right)^{-1}  \right]$  \\
	D-optimal & $\min_{\mathcal{S}_J\subseteq \mathcal{J}} \mathrm{det} \left[ \left( \mathbf{T}^{\alpha,\beta}_{\mathcal{S}_J}\right)^{-1}  \right]$ \\
	E-optimal & $\min_{\mathcal{S}_J \subseteq \mathcal{J} } \left\|  \left(  \left( \mathbf{T}^{\alpha,\beta}\right) ^{1/2}_{\mathcal{S}_J\mathcal{J}}\right) ^{\dag }  \right\| _{2}$   \\
	T-optimal & $\max_{\mathcal{S}_J\subseteq \mathcal{J}} \mathrm{tr} \left[ \left( \mathbf{T}^{\alpha,\beta}_{\mathcal{S}_J}\right)  \right]$ \\
	\bottomrule
\end{tabular}
\end{table}

\subsection{Localized filter representation of optimal JFRFT sampling operators}
If we define $h(\mathbf{\Delta}_J)$ as an ideal low-pass filter, i.e., $\boldsymbol{\Sigma}_{\mathcal{F}_J} = h(\mathbf{\Delta}_J)$, then the following holds, $\mathbf{T}^{\alpha,\beta} = \mathbf{J}^{-\alpha,-\beta} h(\mathbf{\Delta}_J) \mathbf{J}^{\alpha,\beta} = \mathbf{J}^{-\alpha,-\beta} \boldsymbol{\Sigma}_{\mathcal{F}_J} \mathbf{J}^{\alpha,\beta} = \mathbf{B}_J^{\alpha,\beta}$. This key operator relationship enables the localization of optimal sampling operators through the joint fractional filtering operator.

\textit{MaxSigMin}: Based on Eq. \eqref{SJ1}, the optimal sampling set can be reformulated using the localization operator
\[
\begin{aligned}
\mathcal{S}_J^{\text{opt}} = &\argmax_{\mathcal{S}_J \subseteq \mathcal{J}}~ \sigma_{\min} \left( \mathbf{J}^{-\alpha,-\beta}_{\mathcal{S}_J \mathcal{F}_J} \right)
= \argmax_{\mathcal{S}_J \subseteq \mathcal{J}}~ \sigma_{\min} \left( \mathbf{J}^{\alpha,\beta}_{\mathcal{F}_J \mathcal{S}_J} \right) \\
= &\argmax_{\mathcal{S}_J \subseteq \mathcal{J}}~ \sigma_{\min} \left( \mathbf{J}^{-\alpha,-\beta} \boldsymbol{\Sigma}_{\mathcal{F}_J} \mathbf{J}^{\alpha,\beta} \mathbf{D}_{\mathcal{S}_J} \right)
= \argmin_{\mathcal{S}_J \subseteq \mathcal{J}}~ \left\| \left( \mathbf{T}^{\alpha,\beta}_{\mathcal{J} \mathcal{S}_J} \right)^{\dagger} \right\|_2,
\end{aligned}
\]
where, $\mathbf{T}^{\alpha,\beta}$ serves as the spectral-domain localization operator. At iteration $m$, the optimal vertex can be selected by
\begin{equation}
y^{\text{opt}} = \argmin_{y \in \mathcal{J} \setminus \mathcal{S}_{J_m}}~ \left\| \left( \mathbf{T}^{\alpha,\beta}_{\mathcal{J}, \mathcal{S}_{J_m} \cup y} \right)^{\dagger} \right\|_2.
\label{y1}
\end{equation}

\textit{MinTrac}: The trace minimization of the covariance matrix (Eq. \eqref{SJ2}) can be equivalently expressed as
\[
\mathcal{S}_J^{\text{opt}} = \argmin_{\mathcal{S}_J \subseteq \mathcal{J}}~ \mathrm{tr} \left[ \left( \mathbf{J}^{\alpha,\beta}_{\mathcal{S}_J \mathcal{F}_J} \mathbf{J}^{-\alpha,-\beta}_{\mathcal{S}_J \mathcal{F}_J} \right)^{-1} \right]
= \argmin_{\mathcal{S}_J \subseteq \mathcal{J}}~ \mathrm{tr} \left[ \left( \mathbf{T}^{\alpha,\beta}_{\mathcal{S}_J} \right)^{-1} \right].
\]
Correspondingly,
\begin{equation}
y^{\text{opt}} = \argmin_{y \in \mathcal{J} \setminus \mathcal{S}_{J_m}}~ \mathrm{tr} \left[ \left( \mathbf{T}^{\alpha,\beta}_{\mathcal{S}_{J_m} \cup y} \right)^{-1} \right].
\label{y2}
\end{equation}

\textit{MinPinv}: Based on Eq. \eqref{SJ3}, the pseudo-inverse criterion can also be localized as
\[
\begin{aligned}
\mathcal{S}_J^{\text{opt}} =
&\argmin_{\mathcal{S}_J \subseteq \mathcal{J}}~ \sum_{i=1}^{K_J} \frac{1}{\sigma_i \left( \boldsymbol{\Sigma}_{\mathcal{F}_J} \mathbf{J}^{\alpha,\beta} \mathbf{D}_{\mathcal{S}_J} \right)} \\
= &\argmin_{\mathcal{S}_J \subseteq \mathcal{J}}~ \left\| \left( \mathbf{T}^{\alpha,\beta}_{\mathcal{J} \mathcal{S}_J} \right)^{\dagger} \right\|_F
=\argmin_{\mathcal{S}_J \subseteq \mathcal{J}}~ \mathrm{tr} \left[ \left( \mathbf{T}^{\alpha,\beta}_{\mathcal{S}_J} \right)^{-1} \right],
\end{aligned}
\label{S3}
\]
with the vertex selection criterion,
\begin{equation}
y^{\text{opt}} = \argmin_{y \in \mathcal{J} \setminus \mathcal{S}_{J_m}}~ \mathrm{tr} \left[ \left( \mathbf{T}^{\alpha,\beta}_{\mathcal{S}_{J_m} \cup y} \right)^{-1} \right].
\label{y3}
\end{equation}

\textit{MaxSig}: From Eq. \eqref{SJ4}, we obtain
\[
\mathcal{S}_J^{\text{opt}} = \argmax_{\mathcal{S}_J \subseteq \mathcal{J}}~ \left\| \mathbf{T}^{\alpha,\beta}_{\mathcal{J} \mathcal{S}_J} \right\|_F = \argmax_{\mathcal{S}_J \subseteq \mathcal{J}}~ \mathrm{tr} \left[ \mathbf{T}^{\alpha,\beta}_{\mathcal{S}_J} \right].
\label{S4}
\]
Accordingly,
\begin{equation}
y^{\text{opt}} = \argmax_{y \in \mathcal{J} \setminus \mathcal{S}_{J_m}}~ \mathrm{tr} \left[ \mathbf{T}^{\alpha,\beta}_{\mathcal{S}_{J_m} \cup y} \right].
\label{y4}
\end{equation}

\textit{MaxVol}: The volume maximization in Eq. \eqref{SJ5} can be rewritten as
\[
\mathcal{S}_J^{\text{opt}} = \argmax_{\mathcal{S}_J \subseteq \mathcal{J}}~ \det \left[ \mathbf{T}^{\alpha,\beta}_{\mathcal{S}_J} \right],
\label{S5}
\]
and its vertex selection counterpart becomes
\begin{equation}
y^{\text{opt}} = \argmax_{y \in \mathcal{J} \setminus \mathcal{S}_{J_m}}~ \det \left[ \mathbf{T}^{\alpha,\beta}_{\mathcal{S}_{J_m} \cup y} \right].
\label{y5}
\end{equation}

\begin{table}[t]
\caption{Sampling Method with Objective Functions and Localization Operators}
\label{tab02}
\footnotesize
\centering
\begin{tabular}{cll}
	\toprule
	Methods & \ \ \ \ Objective Functions& Localized Filter Operators \\
	\midrule
	MaxSigMin& $\argmax_{\mathcal{S}_J \subseteq \mathcal{J}}~\sigma_{\min} \left( \mathbf{J}^{-\alpha,-\beta}_{\mathcal{S}_J \mathcal{F}_J} \right)$  &$\argmin_{\mathcal{S}_J \subseteq \mathcal{J}}~ \left\| \left( \mathbf{T}^{\alpha,\beta}_{\mathcal{J} \mathcal{S}_J} \right)^{\dagger} \right\|_2$\\
	MinTrac &$\argmin_{\mathcal{S}_J \subseteq \mathcal{J}}~\mathrm{tr} \left[ \left( \mathbf{J}^{\alpha,\beta}_{\mathcal{S}_J\mathcal{F}_J} \mathbf{J}^{-\alpha,-\beta}_{\mathcal{S}_J\mathcal{F}_J} \right)^{-1} \right]$&$\argmin_{\mathcal{S}_J \subseteq \mathcal{J}}~ \mathrm{tr} \left[ \left( \mathbf{T}^{\alpha,\beta}_{\mathcal{S}_J} \right)^{-1} \right].$\\
	MinPinv& $\argmin_{\mathcal{S}_J \subseteq \mathcal{J}}~\sum_{i=1}^{K_J} \frac{1}{\sigma_i\left(\boldsymbol{\Sigma}_{\mathcal{F}_J} \mathbf{J}^{\alpha,\beta} \mathbf{D}_{\mathcal{S}_J}\right) }$&$\argmin_{\mathcal{S}_J \subseteq \mathcal{J}}~ \mathrm{tr} \left[ \left( \mathbf{T}^{\alpha,\beta}_{\mathcal{S}_J} \right)^{-1} \right]$\\
	MaxSig& $\argmax_{\mathcal{S}_J \subseteq \mathcal{J}}~\sum_{i=1}^{K_J} \sigma_i\left(\boldsymbol{\Sigma}_{\mathcal{F}_J} \mathbf{J}^{\alpha,\beta} \mathbf{D}_{\mathcal{S}_J}\right)$ & $\argmax_{\mathcal{S}_J \subseteq \mathcal{J}}~ \mathrm{tr} \left[ \mathbf{T}^{\alpha,\beta}_{\mathcal{S}_J} \right]$\\
	MaxVol& $ \argmax_{\mathcal{S}_J \subseteq \mathcal{J}}~\det \left[\left(  \mathbf{J}^{-\alpha,-\beta} \boldsymbol{\Sigma}_{\mathcal{F}_J} \mathbf{J}^{\alpha,\beta}\right) _{\mathcal{S}_J} \right] $ &$ \argmax_{\mathcal{S}_J \subseteq \mathcal{J}}~ \det \left[ \mathbf{T}^{\alpha,\beta}_{\mathcal{S}_J} \right]$\\
	\bottomrule
\end{tabular}
\end{table}

Eqs. \eqref{y1}–\eqref{y5} unify the five JFRFT sampling strategies within the localized filter operator framework $\mathbf{T}^{\alpha,\beta}$, with each sampling operator design relying on the spectral kernel $h(\mathbf{\Delta}_J)$. The selected spectral kernel corresponds to the spectral bandlimited operator, bridging the sampling operators and their localized forms. A summary of the corresponding objective functions is given in Table~\ref{tab02}.

\section{Numerical experiments}
\label{Experiments}
This section presents numerical experiments on multiple datasets \cite{JFT,JFRFT,JLCT} to evaluate the performance of JFRFT-based sampling strategies. The seasonal sunshine dataset across U.S. states is used to compare the sampling set selections of five representative methods. The dog walking dynamic mesh dataset with varying node counts is employed to assess computational efficiency and reconstruction accuracy. To examine the impact of different JFRFT parameters $(\alpha, \beta)$, time-series sea clutter signals are utilized. All experiments are conducted in MATLAB R2024a using the GSP Toolbox \cite{GSPBox}.

\subsection{Comparison of sampling set selection}
To investigate the impact of JFRFT on sampling set selection, an experimental setup was constructed using seasonal sunshine duration data across U.S. states, following the structure in Fig. \ref{fig01} . A joint time-vertex graph was formed as the 192-node Cartesian product graph \cite{HGFRFT}. A joint bandlimited signal was defined with time and vertex bandwidths $K_T = 2$ and $K_G = 10$, giving a total joint bandwidth $K_J = 20$. The JFRFT operator was constructed with fixed parameters $\alpha = 0.7$, $\beta = 0.8$, and the original signal was projected onto this subspace to obtain the bandlimited signal.

To preliminarily assess the effect of JFRFT parameters on sampling and reconstruction, $M_J = 20$ sampling nodes were selected using five different strategies. Gaussian noise $\mathcal{N}(0, 0.01^2)$ was added to the observations, and signal reconstruction was performed using Theorem \ref{thm3}. Each experiment was repeated 10 times with independently drawn noise, and the average NMSE was recorded. The best result was achieved by the MaxSigMin method with parameters $(2.2, 0.8)$, yielding a minimal NMSE of $1.61 \times 10^{-6}$, as shown in Fig. \ref{fig03}, where (a) shows the original bandlimited signal and (b) the reconstruction using the optimal sampling strategy.
\begin{figure}[h]
\begin{center}
	\begin{minipage}[t]{0.45\linewidth}
		\centering
		\includegraphics[width=\linewidth]{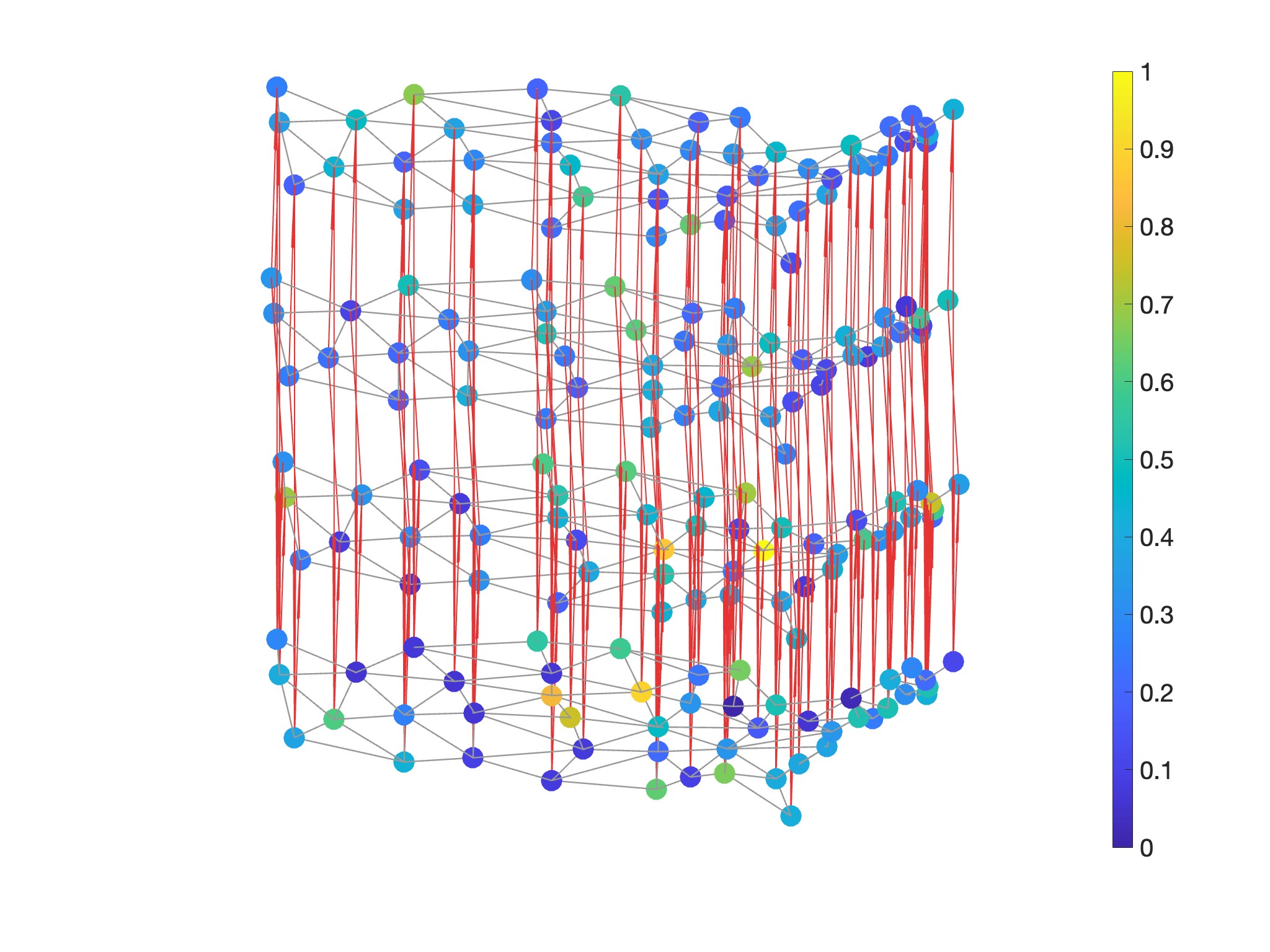}
		\parbox{4cm}{\tiny (a) Original bandlimited sunshine signals.}
	\end{minipage}
	\begin{minipage}[t]{0.45\linewidth}
		\centering
		\includegraphics[width=\linewidth]{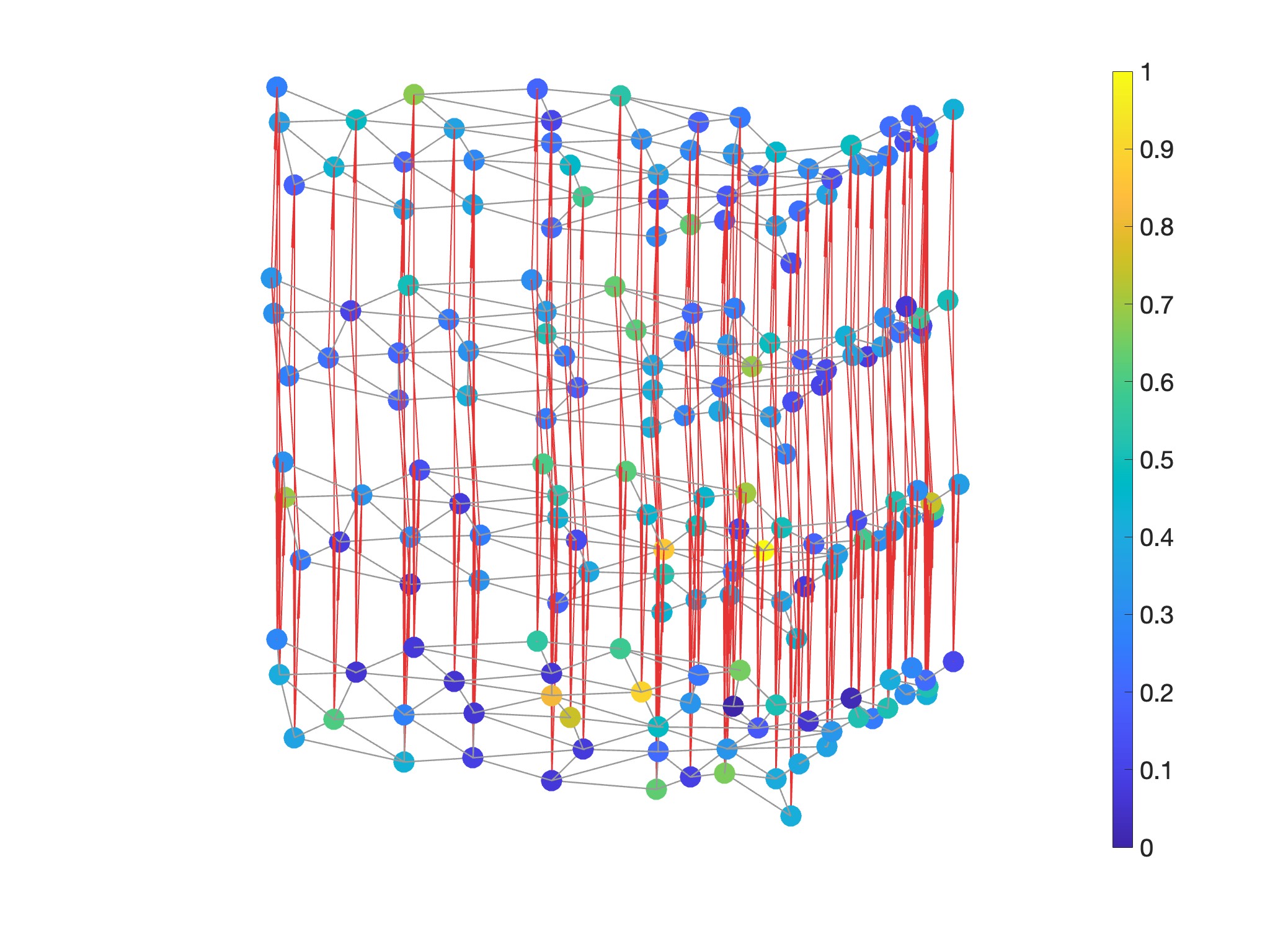}
		\parbox{4.5cm}{\tiny (b) Recovery signals using the MaxSigMin method.}
	\end{minipage}
\end{center}
\caption{Original and recovery time-vertex graph signals.}
\vspace*{-3pt}
\label{fig03}
\end{figure}

Fig. \ref{fig04} visualizes the optimal sampling locations (20 nodes) selected by each method. Notably, MaxSigMin, MinTrac, MinPinv, and MaxVol generally tend to select sampling points that are well distributed across the spatial domain. Among them, MinTrac and MinPinv exhibit fundamentally consistent selection behavior, resulting in identical sampling sets despite differences in their computational runtime. In contrast, MaxSig often fails to account for previously selected nodes during the iterative process, leading to the selection of spatially redundant vertices in close proximity. This redundancy may degrade reconstruction performance due to insufficient spatial coverage. However, the severity of this error can vary with different JFRFT parameters, suggesting that appropriate tuning of $(\alpha, \beta)$ can help mitigate the impact of suboptimal sampling configurations.
\begin{figure}[h]
\begin{center}
	\begin{minipage}[t]{0.9\linewidth}
		\centering
		\includegraphics[width=\linewidth]{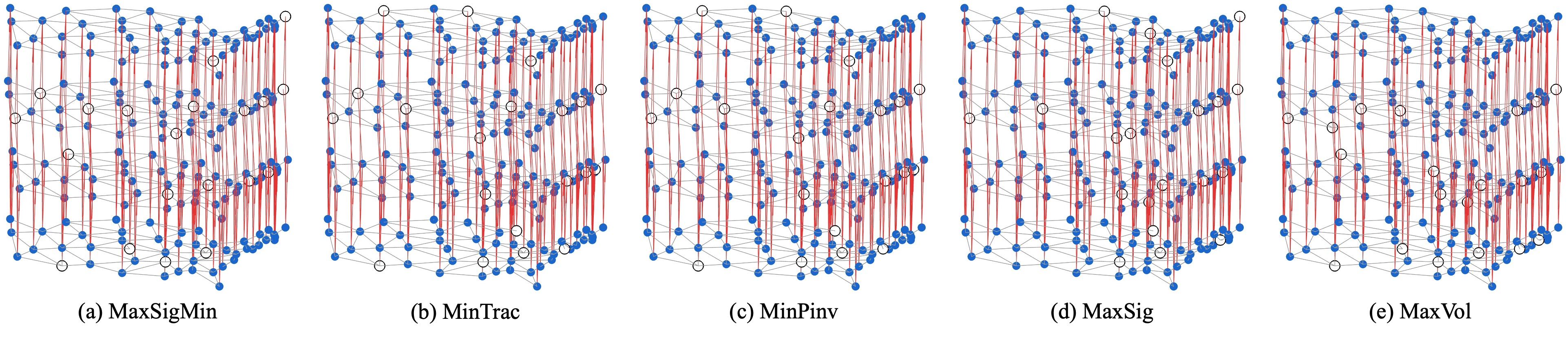}
	\end{minipage}
\end{center}
\vspace*{-10pt}
\caption{Optimal sampling locations (white nodes) selected by five sampling methods.}
\vspace*{-3pt}
\label{fig04}
\end{figure}

\subsection{Comparison of recovery time}
To assess the computational efficiency and reconstruction performance of various sampling strategies under different joint graph sizes, a series of experiments was conducted within the JFRFT framework. The joint graph was constructed using the first 10 columns (time dimension) of the dog walking dataset \cite{JFT,JFRFT} and a spatial graph formed by selecting 40 to 180 rows. Each configuration yielded $N \times T$ joint nodes, with $T = 10$ and $N$ ranging from 40 to 180. The sampling budget was set to $M_J = NT / 10$. Bandlimited signals were generated using the joint operator $\mathbf{B}^{\alpha,\beta}_J$ with fixed JFRFT parameters $\alpha = 0.8$, $\beta = 0.9$, and joint bandwidths $K_T = 2$ and $K_G = 20$.

Five sampling methods were evaluated, with random sampling as a baseline. Each method selected $M_J$ nodes using a localization matrix $\mathbf{T}^{\alpha,\beta}$ aligned with the joint bandlimited operator $\mathbf{B}^{\alpha,\beta}_J$. To simulate noise, Gaussian perturbations $\mathcal{N}(0, 0.05^2)$ were added to the sampled signals. Reconstruction was performed, and both runtime and NMSE were recorded. Each setting was repeated ten times, and average results are reported.

Figure~\ref{fig05} compares runtime and NMSE across varying joint graph sizes. In panel (a), MaxSigMin, MinPinv, and MaxSig show relatively low computational cost with competitive performance, while MinTrace and MaxVol incur higher runtime as graph size grows. Panel (b) shows MinTrace, MinPinv, and MaxVol consistently yield the lowest NMSE, with MaxSigMin and MaxSig slightly behind. Random sampling performs worst overall. Therefore, MinPinv provides the best trade-off between efficiency and accuracy among the methods considered.
\begin{figure}[h]
\begin{center}
	\begin{minipage}[t]{0.45\linewidth}
		\centering
		\includegraphics[width=\linewidth]{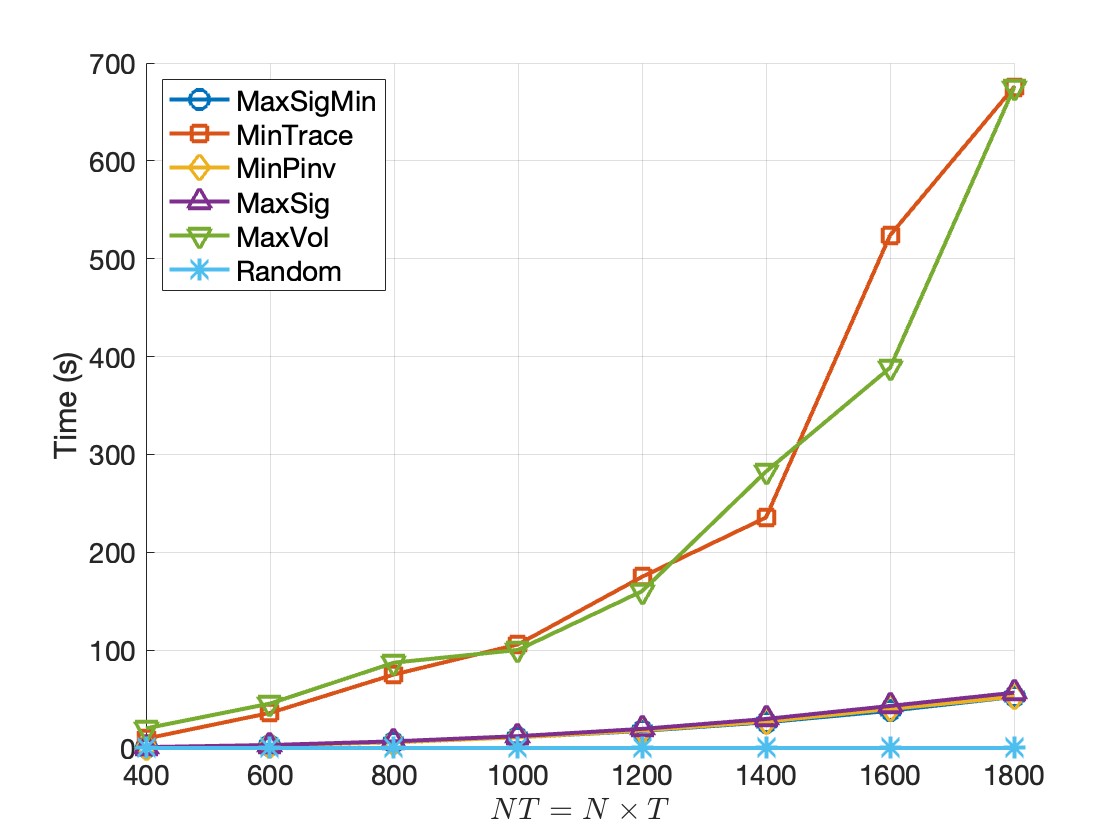}
		\parbox{5cm}{\tiny (a) Comparison of runtime across varying joint graph sizes.}
	\end{minipage}
	\begin{minipage}[t]{0.45\linewidth}
		\centering
		\includegraphics[width=\linewidth]{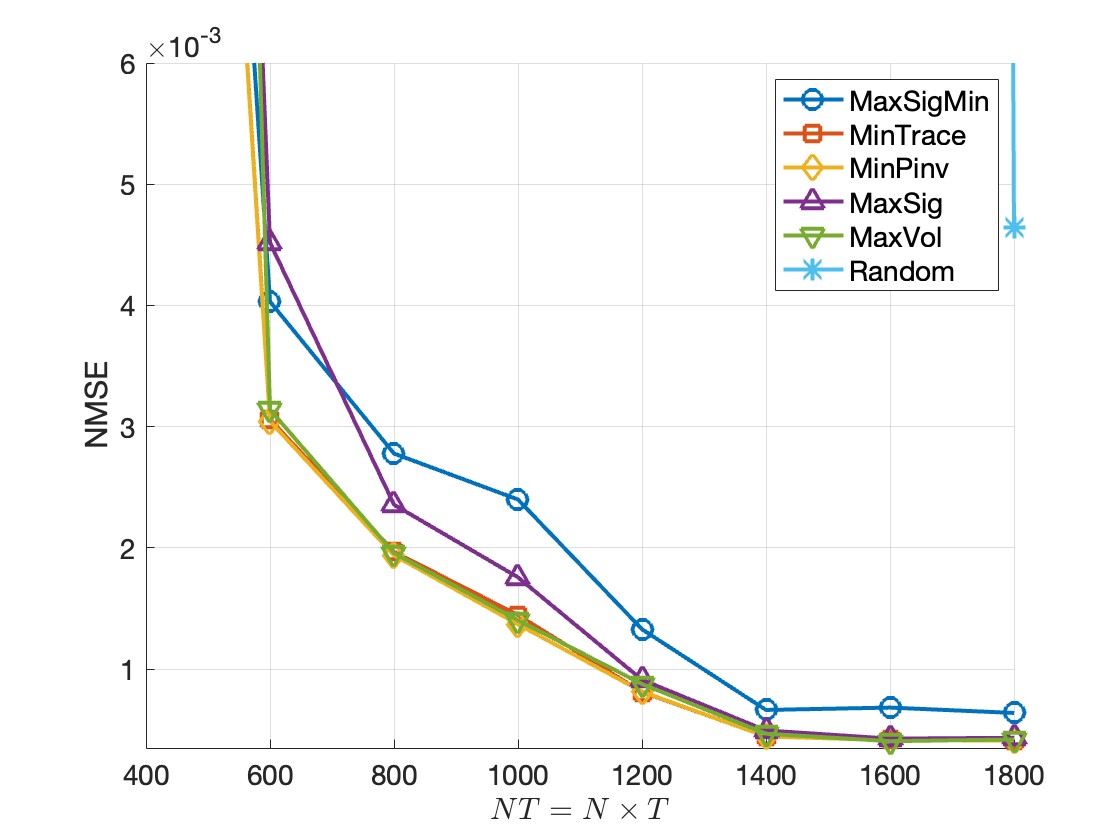}
		\parbox{6cm}{\tiny (b) Comparison of corresponding NMSE for different joint graph sizes.}
	\end{minipage}
\end{center}
\vspace*{-10pt}
\caption{Comparison of running time and NMSE for different joint graph sizes based on sampling methods.}
\vspace*{-3pt}
\label{fig05}
\end{figure}

\subsection{Comparison of recovery errors with different JFRFT parameters}
To investigate the impact of JFRFT parameters on signal sampling and reconstruction performance, we conduct systematic experiments using a real-world sea clutter dataset\footnote{Available: http://soma.mcmaster.ca/ipix/dartmouth/cdf051100.html}. The raw data, obtained from a spaceborne radar system, comprises complex-valued signals from a specific range cell. Preprocessing yields a time-vertex signal with $N=14$ vertices and $T=10$ time steps. Based on the known sea clutter signals, we construct the underlying vertex graph and temporal graph using Gaussian kernels, thereby forming a joint time-vertex representation of the signal. The associated joint localization operator $\mathbf{T}^{\alpha,\beta}$ is then constructed with $K_J = 10$. Notably, we use the raw data directly without applying any bandlimited preprocessing.

To select the JFRFT parameters, we perform a two-stage search: an initial coarse scan followed by a fine-grained search within the range $\alpha, \beta \in [-4, 4]$. For each parameter pair, $M_J=10$ sampling nodes are identified via the bandlimited operator, with the MinPinv method used to select the optimal subset. The original signal is then reconstructed, and the reconstruction error is computed. As shown in Fig.~\ref{fig06}, 
\begin{figure}[h]
\begin{center}
	\begin{minipage}[t]{0.8\linewidth}
		\centering
		\includegraphics[width=\linewidth]{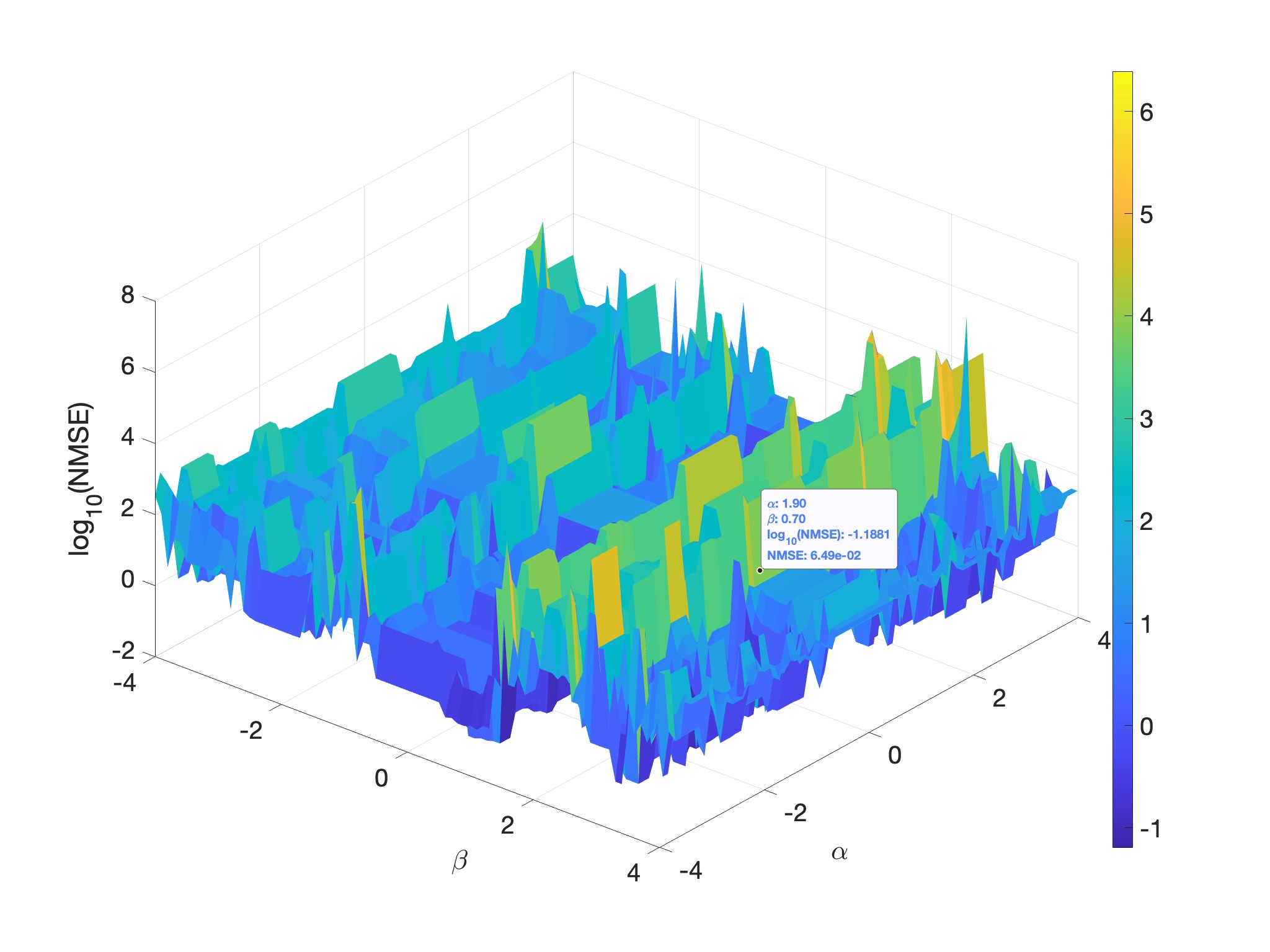}
	\end{minipage}
\end{center}
\vspace*{-20pt}
\caption{NMSE vs. MinPinv for different JFRFT parameters.}
\vspace*{-3pt}
\label{fig06}
\end{figure}
we plot the error surface with $\log_{10}(\text{NMSE})$ as a function of $\alpha$ and $\beta$, revealing that reconstruction performance varies significantly with different parameter choices. A clear optimum is observed at $(\alpha, \beta) = (1.9, 0.7)$, where the minimum NMSE reaches 0.0649. This result highlights the critical role of JFRFT parameters in determining sampling and reconstruction accuracy.

Finally, we analyze the statistical characteristics of the reconstructed signal under the optimal JFRFT setting. We compare it against the JFT-based reconstruction, as well as Rayleigh and Weibull distribution models, by fitting their probability density functions (PDF) and cumulative distribution functions (CDF). As shown in Fig.~\ref{fig07}, this comparison provides insights into the statistical nature of the recovered signal, offering potential guidance for future applications such as target detection or clutter suppression based on reconstruction outputs.
\begin{figure}[h]
\begin{center}
	\begin{minipage}[t]{0.45\linewidth}
		\centering
		\includegraphics[width=\linewidth]{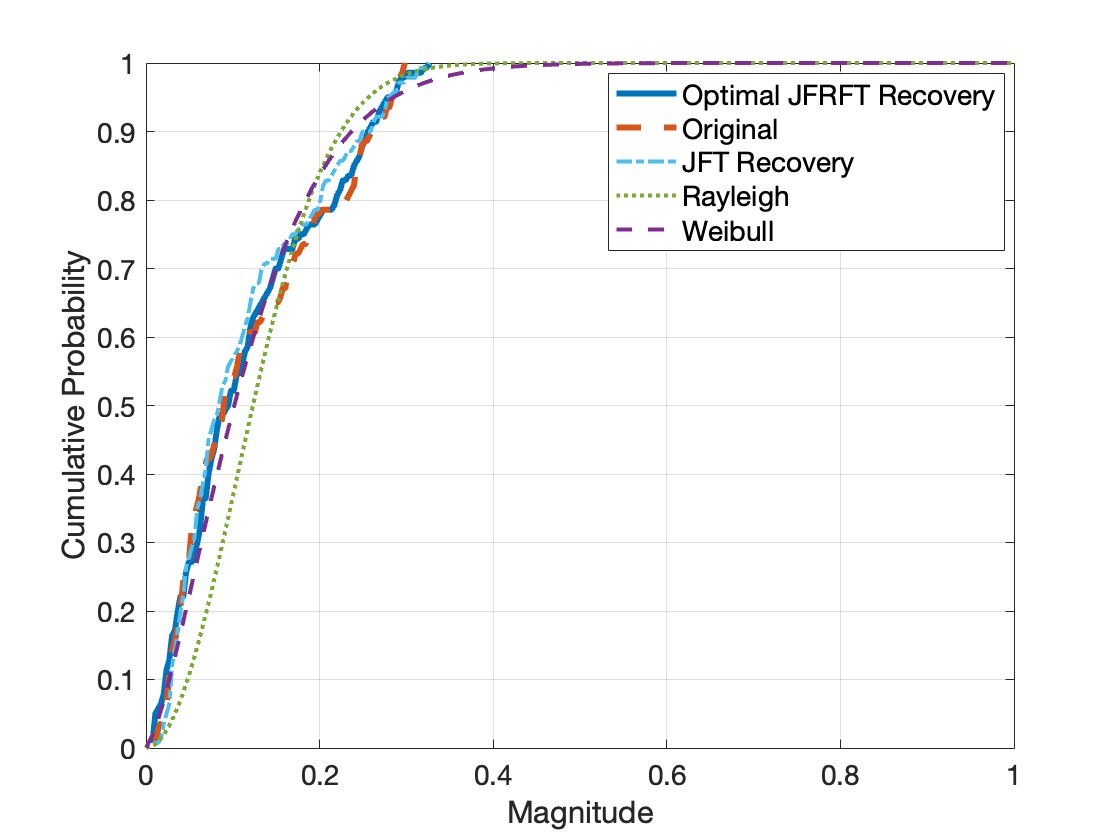}
		\parbox{1cm}{\tiny (a) CDF.}
	\end{minipage}
	\begin{minipage}[t]{0.45\linewidth}
		\centering
		\includegraphics[width=\linewidth]{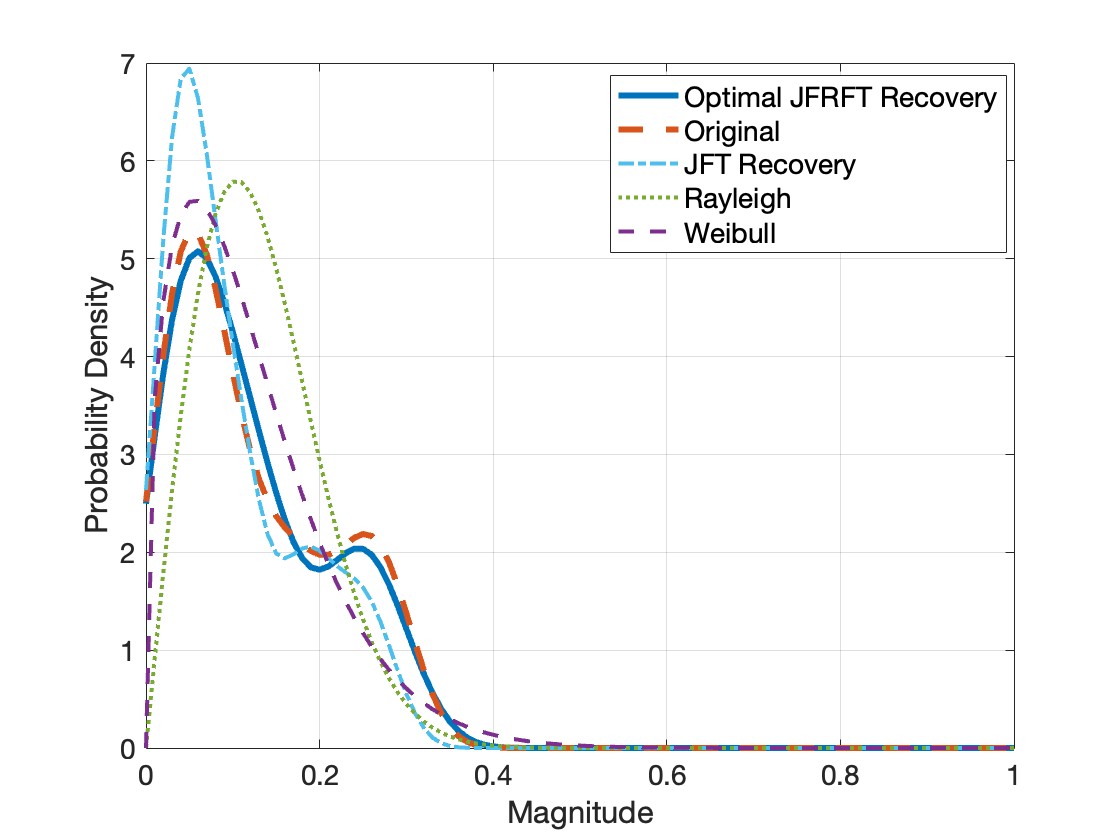}
		\parbox{1cm}{\tiny (b) PDF.}
	\end{minipage}
\end{center}
\caption{Comparison of statistical characteristics of various methods.}
\vspace*{-3pt}
\label{fig07}
\end{figure}

\section{Conclusion}
\label{Conclusion}
In this work, we propose a novel sampling theory in the JFRFT domain, enhancing the expressive power of signal representations in time-frequency analysis and GSP. By establishing the conditions for perfect recovery of jointly bandlimited signals and developing optimal sampling set selection strategies, we construct a unified JFRFT-based sampling framework. Multiple sampling strategies are proposed and systematically compared within this framework. To further improve the efficiency of large-scale joint time-vertex signal processing, we investigate the design of localized sampling operators and integrate them with the proposed strategies. Numerical experiments demonstrate the superior reconstruction accuracy and computational efficiency of our methods, offering a new paradigm for efficient processing of time-varying graph signals.

\appendix
\section{Proof of Theorem 1}
\label{AA}
Suppose $\bm{x} \in \mathbb{C}^{NT}$ satisfies $\mathbf{B}_J^{\alpha,\beta} \bm{x} = \bm{x}$ and $\mathbf{D}_J \bm{x} = \bm{x}$. Then,
\[
\mathbf{B}_J^{\alpha,\beta} \mathbf{D}_J \mathbf{B}_J^{\alpha,\beta} \bm{x}
= \mathbf{B}_J^{\alpha,\beta} \mathbf{D}_J \bm{x}
= \mathbf{B}_J^{\alpha,\beta} \bm{x}
= \bm{x},
\]
showing that $\bm{x}$ is an eigenvector of $\mathbf{B}_J^{\alpha,\beta} \mathbf{D}_J \mathbf{B}_J^{\alpha,\beta}$ associated with eigenvalue 1, and hence $\lambda_{\max} = 1$.

Conversely, assume
\[
\mathbf{B}_J^{\alpha,\beta} \mathbf{D}_J \mathbf{B}_J^{\alpha,\beta} \bm{x} = \bm{x}.
\]
Using the idempotent and Hermitian properties of the projectors, we write
\[
\mathbf{B}_J^{\alpha,\beta} \mathbf{D}_J \mathbf{B}_J^{\alpha,\beta} \bm{x} 
= \mathbf{B}_J^{\alpha,\beta} \mathbf{D}_J  (\mathbf{B}_J^{\alpha,\beta})^{2} \bm{x},
\]
which implies $\mathbf{B}_J^{\alpha,\beta} \bm{x} = \bm{x}$. Then, by the Rayleigh–Ritz theorem, the spectral norm condition
\[
\max_{\bm{x} \neq 0}
\frac{ \bm{x}^* (\mathbf{D}_J \mathbf{B}_J^{\alpha,\beta}) \bm{x} }
{ \| \bm{x} \|^2 } = 1
\]
implies $\mathbf{D}_J \bm{x} = \bm{x}$. Therefore, $\bm{x}$ lies in the intersection of the joint support and frequency subspaces.

Finally, converting back to matrix form yields
\[
\mathbf{B}_G^\beta \mathbf{X} \mathbf{B}_T^\alpha = \mathbf{X}, \quad
\mathbf{D}_G \mathbf{X} \mathbf{D}_T = \mathbf{X}.
\]

\section{Proof of Theorem 2}
\label{AB}
Let $\bm{x} \in \text{BL}_{K_J}(\mathbf{J}^{\alpha,\beta})$, and consider the sampling operator $\mathbf{D}_J$ and the reconstruction operator $\mathbf{R}_J$. Suppose  
\[
\text{rank}(\mathbf{D}_J|_{\text{BL}_{K_J}}) = \text{rank}(\mathbf{R}_J \mathbf{D}_J|_{\text{BL}_{K_J}}) = K_J,
\]  
which ensures that $\mathbf{R}_J$ spans the entire $\alpha,\beta$-bandlimited space $\text{BL}_{K_J}(\mathbf{J}^{\alpha,\beta})$.

Define $\mathbf{O} := \mathbf{R}_J \mathbf{D}_J$. Since both $\mathbf{R}_J$ and $\mathbf{D}_J$ are supported on $\text{BL}_{K_J}(\mathbf{J}^{\alpha,\beta})$, and $\mathbf{D}_J \mathbf{R}_J = \mathbf{I}$ on the image of $\mathbf{D}_J$, we have
\[
\mathbf{O}^2 = \mathbf{R}_J \mathbf{D}_J \mathbf{R}_J \mathbf{D}_J = \mathbf{R}_J (\mathbf{D}_J \mathbf{R}_J) \mathbf{D}_J = \mathbf{R}_J \mathbf{D}_J = \mathbf{O},
\]  
which shows that $\mathbf{O}$ is a projection operator onto $\text{BL}_{K_J}(\mathbf{J}^{\alpha,\beta})$.

Therefore, for any $\bm{x} \in \text{BL}_{K_J}(\mathbf{J}^{\alpha,\beta})$, we have
\[
\bm{x} = \mathbf{O} \bm{x} = \mathbf{R}_J \mathbf{D}_J \bm{x} = \mathbf{R}_J \bm{x}_{\mathcal{S}_J}.
\]

\section{Proof of Theorem 3}
\label{AC}

The reconstruction $\bm{x}' $ can be expressed as
\begin{equation*}
	\begin{aligned}
		\bm{x}' 
		=&\mathbf{T}^{\alpha,\beta}_{\mathcal{J}\mathcal{S}_J} \left( \mathbf{T}^{\alpha,\beta}_{\mathcal{S}_J} \right)^{\dag} \bm{x}_{\mathcal{S}_J} \\
		=& \left[ \mathbf{J}^{-\alpha,-\beta} h\left( \mathbf{\Delta}_J \right) \mathbf{J}^{\alpha,\beta} \right]_{\mathcal{J}\mathcal{S}_J} \left( \left[ \mathbf{J}^{-\alpha,-\beta} h\left( \mathbf{\Delta}_J \right) \mathbf{J}^{\alpha,\beta} \right]_{\mathcal{S}_J} \right)^{\dag} \bm{x}_{\mathcal{S}_J} \\
		=& \mathbf{J}^{-\alpha,-\beta} h\left( \mathbf{\Delta}_J \right) \mathbf{J}^{\alpha,\beta}_{\mathcal{S}_J \mathcal{J}} \left( \mathbf{J}^{-\alpha,-\beta}_{\mathcal{S}_J \mathcal{J}} h\left( \mathbf{\Delta}_J \right) \mathbf{J}^{\alpha,\beta}_{\mathcal{S}_J \mathcal{J}} \right)^{\dag} \bm{x}_{\mathcal{S}_J} \\
		=& \mathbf{J}^{-\alpha,-\beta} h^{\frac{1}{2}} \left( \mathbf{\Delta}_J \right) \left( \mathbf{J}^{-\alpha,-\beta}_{\mathcal{S}_J \mathcal{J}} h^{\frac{1}{2}} \left( \mathbf{\Delta}_J \right) \right)^{\dag} \bm{x}_{\mathcal{S}_J}.
	\end{aligned}
\end{equation*}

This expression can be equivalently reformulated as
\begin{equation}
	\bm{x}' = \mathbf{J}^{-\alpha,-\beta} \left( \frac{h\left( \mathbf{\Delta}_J \right)}{\rho} \right)^{\frac{1}{2}} \left( \mathbf{J}^{-\alpha,-\beta}_{\mathcal{S}_J \mathcal{J}} \left( \frac{h\left( \mathbf{\Delta}_J \right)}{\rho} \right)^{\frac{1}{2}} \right)^{\dag} \bm{x}_{\mathcal{S}_J}, \label{xR}
\end{equation}
where $\rho = \min_{0 \leq i \leq K_J - 1} [h(\mathbf{\Delta}_J)]_{ii}$. When \( i \geq K_J \), the corresponding components of \( \left( h\left( \mathbf{\Delta}_J \right)/\rho \right)^{1/2} \) asymptotically vanish. Under this assumption, the following approximation holds
\begin{equation}
	\left( \mathbf{J}^{-\alpha,-\beta}_{\mathcal{S}_J \mathcal{J}} \left( \frac{h\left( \mathbf{\Delta}_J \right)}{\rho} \right)^{\frac{1}{2}} \right)^{\dag} \approx \left( \frac{h\left( \left[\mathbf{\Delta}_J\right]_{\mathcal{F}_J} \right)}{\rho} \right)^{-\frac{1}{2}} \cdot \left( \mathbf{J}^{-\alpha,-\beta}_{\mathcal{S}_J \mathcal{F}_J} \right)^{\dag}. \label{approximation}
\end{equation}

Substituting \eqref{xR} and \eqref{approximation} yields
\[
\begin{aligned}
	\bm{x}' 
	=& \mathbf{J}^{-\alpha,-\beta}_{\mathcal{J} \mathcal{F}_J} \left( \mathbf{J}^{-\alpha,-\beta}_{\mathcal{S}_J \mathcal{F}_J} \right)^{\dag} \bm{x}_{\mathcal{S}_J} = \mathbf{J}^{-\alpha,-\beta}_{\mathcal{J}} \mathbf{\Sigma}_{\mathcal{F}_J} \mathbf{J}^{\alpha,\beta}_{\mathcal{S}_J \mathcal{J}} \bm{x}_{\mathcal{S}_J} \\
	=&\mathbf{D}_{\mathcal{S}_J} \mathbf{J}^{-\alpha,-\beta}_{\mathcal{J}} \mathbf{\Sigma}_{\mathcal{F}_J} \mathbf{J}^{\alpha,\beta}_{\mathcal{J}} \bm{x}_{\mathcal{S}_J} = \mathbf{D}_J \mathbf{B}^{\alpha,\beta}_J \bm{x}_{\mathcal{S}_J} = \mathbf{R}_J \bm{x}_{\mathcal{S}_J},
\end{aligned}
\]
which completes the proof.

\section*{Declaration of competing interest}
The authors declare that they have no known competing financial interests or
personal relationships that could have appeared to influence the work reported
in this paper.

\section*{Acknowledgments}
This work were supported by the National Natural Science Foundation of China [No. 62171041] and Natural Science Foundation of Beijing Municipality [No. 4242011].


\end{document}